\documentclass[oneside,11pt,reqno]{amsart}
\usepackage{tikz}
\usepackage{amsmath,amsthm,amssymb,amsfonts,mathrsfs,bbm,graphicx,array,bbm,pgffor,extarrows,tikz-cd,romanbar}
\usepackage[utf8]{inputenc}
\usepackage[makeroom]{cancel}
\usepackage{geometry} 
\usepackage{mathtools}
\usepackage{mathdots}
\usepackage{enumerate}
\usepackage[mathscr]{euscript}
\usepackage{csquotes}
\MakeOuterQuote{"}

\usepackage[backend = biber, style=numeric, citestyle=numeric, giveninits=true, eprint=false
]{biblatex}
	\renewbibmacro{in:}{}
	\AtEveryBibitem{\clearlist{language}}
        \AtEveryBibitem{\clearfield{urlyear}}
        
\usepackage[colorlinks=true,allcolors=blue]{hyperref} 
\usepackage{xurl}
\usepackage{cleveref}

\DeclareMathOperator{\Id}{Id}

\DeclareSymbolFont{rsfs}{U}{rsfs}{m}{n}
\DeclareSymbolFontAlphabet{\mathscrsfs}{rsfs}
\newcommand{\C}{\mathbb{C}}

\newcommand{\Fosp}{\mathfrak{osp}}
\newcommand{\Fso}{\mathfrak{so}}
\newcommand{\Fh}{\mathfrak{h}}

\newcommand{\Fn}{\mathfrak{n}}
\newcommand{\ga}{\gamma}

\makeatletter
\newcommand\bigDiamond{\mathop{\mathpalette\bigDi@mond\relax}}
\newcommand\bigDi@mond[2]{%
  \vcenter{\hbox{\m@th
    \scalebox{\ifx#1\displaystyle 2\else1.2\fi}{$#1\Diamond$}%
  }}%
}
\newcommand\bigLozenge{\mathop{\mathpalette\bigL@zenge\relax}}
\newcommand\bigL@zenge[2]{%
  \vcenter{\hbox{\m@th
    \scalebox{\ifx#1\displaystyle 2\else1.2\fi}{$#1\blacklozenge$}%
  }}%
}
\makeatother

\newcommand{\II}{\textnormal{\Romanbar{2}}}
\newcommand{\ad}{\operatorname{ad}}

\newcommand{\bDiamond}{\mathbin{\Diamond}}

\theoremstyle{plain}
\newtheorem{thm}{Theorem}[section]
\newtheorem{lemma}[thm]{Lemma}
\newtheorem{cor}[thm]{Corollary} 
\newtheorem{prop}[thm]{Proposition}

\theoremstyle{definition}
\newtheorem{definition}[thm]{Definition}

\newtheorem{notation}[thm]{Notation}

\newtheorem{remark}[thm]{Remark}

\numberwithin{equation}{section}

\newcommand{\catname}[1]{\mathbf{#1}}
\newcommand{\SVect}{\catname{SVect}}

\title{Dirac Reduction Algebra}
\author{Matthew Dorang \and Jonas T. Hartwig \and Dwight Anderson Williams II}
\date{}

\address{Department of Mathematics, Iowa State University, Ames, IA-50011, USA}
\email{mtdorang@iastate.edu}
\urladdr{https://math.iastate.edu/profiles/matthew-dorang/}
\address{Department of Mathematics, Iowa State University, Ames, IA-50011, USA}
\email{jth@iastate.edu}
\urladdr{http://jthartwig.net}
\thanks{J.T.H. and M.D. were supported in part by the Army Research Office grant W911NF-24-1-0058. D.A.W. II was supported by the National Science Foundation under Award No.~2503008.}
\address{Department of Mathematics, Morgan State University, Baltimore, MD-21251, USA}
\email{dwight@mathdwight.com}
\urladdr{https://mathdwight.com}

\begin{document}

\begin{abstract}
There is a homomorphism of associative superalgebras from the enveloping algebra of the orthosymplectic Lie superalgebra $\mathfrak{osp}(1|2)$ to the Weyl-Clifford superalgebra $A(2n|n)$ with $2n$ even Weyl algebra generators and $n$ odd Clifford algebra generators. Under this homomorphism, the positive odd root vector $x\in\mathfrak{osp}(1|2)$ is sent to the Dirac operator $\gamma^\mu\partial_\mu\in A(2n|n)$ and generates a left ideal $I$. The corresponding reduction (super)algebra, denoted $Z_n$, is the normalizer of $I$ in $A(2n|n)$ modulo $I$. By construction, $Z_n$ acts on the space of all Clifford algebra-valued polynomial solutions to the (massless) Dirac equation. This setting goes back to Zhelobenko, who described generators of this reduction algebra and commutation relations among them. In this paper, we find a complete presentation of (a localization of) this so-termed Dirac reduction algebra, with full proofs. We give a closed formula for products of the raising operators in the reduction algebra and use it to write down explicitly the polynomial solutions to the Dirac equation in $n$-dimensional flat spacetime that they generate. Finally, we show that for $n\neq 2$ the space of polynomial solutions is irreducible over the reduction algebra tensored with the opposite Clifford algebra when $n$ is even, and is a direct sum of two irreducible submodules when $n$ is odd.
\end{abstract}

\maketitle

\section{Introduction}

The Dirac equation is the classical equation of motion for the field of a freely propagating spin $1/2$ fermion. As such, it plays a critical role in the standard model of particle physics, being part of the description of all leptons and quarks. The massless Dirac equation is often sufficient for fast moving low-mass particles like the electron and the neutrinos. 
Solutions are often taken to be plane waves, linear combinations of complex exponential functions.

In mathematics, the massless Dirac equation plays a crucial role in Clifford analysis, where it is viewed as a generalization of the Cauchy-Riemann equations; see \cite{BrackxDS1982,ryanCliffordAnalysis2011} for reference.

In 1985, Howe \cite{howeDualPairsPhysics1985} pointed out that there is a connection between several field equations in physics and dual pairs of Lie superalgebras. In particular, the oscillator representation of the Lie superalgebra $\Fosp(2m|2n)$ restricts to a representation of $\Fosp(1|2)$ by differential operators in which the negative odd simple root (with respect to the distinguished root system) is mapped to the Dirac operator. After an automorphism, we can instead map the positive odd simple root vector to the Dirac operator and move in a direction making contact with highest weight theory for the Lie superalgebra $\Fosp(1|2)$. In particular, the solutions space carries the structure of a representation of a so-called reduction algebra as studied in  \cite{mickelssonStepAlgebrasSemisimple1973,vandenhomberghNoteMickelssonsStep1975,khoroshkinDiagonalReductionAlgebras2010,debieHarmonicTransvectorAlgebra2017,hartwigGhostCenterRepresentations2023,hartwigSymplecticDifferentialReduction2025,hartwigExactSolutionsKleinGordon2025,mudrovMickelssonAlgebrasHasse2025}. This approach goes back to Zhelobenko. In his book \cite[Section~7.5]{zhelobenkoRepresentationsReductiveLie1994}, he shows that the Dirac equation is the highest weight condition for $\Fosp(1|2)$, describes generators of the corresponding reduction algebra and commutation relations among them, and shows that every polynomial solution is obtained by applying raising operators to constant solutions; see also \cite{zhelobenkoHypersymmetriesExtremalEquations1997}. The goal of this paper is to give a complete and detailed account of this reduction algebra and of its action on the solutions.

To explain the situation in more detail, the (massless) Dirac equation in $n$-dimensional spacetime with metric $\eta_{ab}$ reads
\begin{equation}\label{eq:Dirac-eq}
    \gamma^i \partial_i \varphi = 0.
\end{equation}
Here summation over $i=1,2,\ldots,n$ is implied, and $\gamma^i$ are generators of the Clifford algebra $C(n) = \mathop{Cl}(n)=\mathop{Cl}(\C^n;\eta)$ satisfying the relations $\{\ga^a,\ga^b\}=2\eta^{ab}$. The function $\varphi=\varphi(x^1,x^2,\ldots,x^n)$ must take values in a space on which the Clifford algebra acts. We consider the case of the regular representation. In any case, the regular representation contains all irreducible representations\footnote{There is only one for $n$ even, and there are two for $n$ odd.} of the Clifford algebra since the latter is semisimple.
We will furthermore restrict attention to polynomial solutions to the Dirac equation \eqref{eq:Dirac-eq}. Such solutions are also called \emph{monogenic polynomials}.

Thus we consider the vector space
\begin{equation}
    V = C(n)\otimes\C[x^1,x^2,\ldots,x^n]
\end{equation}
which carries a natural structure as a module over the Weyl-Clifford algebra
\begin{equation}
    A(2n|n)=A(2n)\otimes C(n),
\end{equation}
where $A(2n)$ is the Weyl algebra with $2n$ generators $x^1,x^2,\ldots,x^n$, $\partial_1, \partial_2,\ldots,\partial_n$ and defining relations $[\partial_i,x^j]=\delta_i^j$, $[x^i,x^j]=0$, $[\partial_i,\partial_j]=0$. The generators $x^i$ and $\ga^i$ act as multiplication operators on $V$, while the $\partial_i$ act as $C(n)$-linear differential operators.
We identify $A(2n)$ and $C(n)$ with their respective image in the tensor product. Therefore, we can read the left hand side of \eqref{eq:Dirac-eq} as an element of $A(2n|n)$. The solution space is thus
\begin{equation}
V^+=\{\varphi\in V\mid \ga^a \partial_a \varphi=0\}.
\end{equation}
Regarding the $\ga^i$ as odd, the $x^i$ and $\partial_i$ as even, there is a homomorphism of associative superalgebras
\begin{equation}
    \psi:\Fosp(1|2)\to A(2n|n)
\end{equation}
which, among other things, sends the odd simple root to (a nonzero multiple of) the Dirac operator.
Writing $\Fosp(1|2)=\Fn_-\oplus\Fh\oplus\Fn_+$ for the triangular decomposition, we can therefore equivalently express the solution space $V^+$ as the space of highest weight vectors (sometimes called singular or primitive vectors)
\begin{equation}
    V^+ = \{\varphi\in V\mid \Fn_+\cdot \varphi=0\},
\end{equation}
which explains our notation $V^+$.

We can therefore utilize the theory of reduction algebras. Consider the left ideal $I=A(2n|n)\cdot \Fn_+$ and let $N=\{a\in A(2n|n)\mid Ia\subset I\}$ be the normalizer of $I$ in $A(2n|n)$. The reduction algebra (or generalized Mickelsson algebra) is defined as the quotient algebra $S=N/I$. The key point is that $S$ acts on $V^+$. That is to say, the space of polynomial solutions to the massless Dirac equation carries a representation of the reduction algebra $S$.

The physical property of Lorentz invariance of the Dirac equation is mathematically expressed by the fact that the solution space carries a representation of the Lie algebra $\Fso(\eta)=\Fso(p,q)$, where $(p,q)$ is the signature of the metric $\eta_{ab}$. Howe \cite{howeDualPairsPhysics1985} observed that this is related to the fact that $\Fosp(1|2)$ and $\Fso(\eta)$ form a dual pair inside $\Fosp(n|2n)$. 

We wish to point out that $\Fso(\eta)$ is in fact a Lie subalgebra of the reduction algebra $S$. In this sense, the algebra $S$ is a larger symmetry algebra that properly contains the Lorentz algebra. Perhaps this is the reason for Zhelobenko's attempt at naming these algebras \emph{hypersymmetry algebras} in \cite{zhelobenkoHypersymmetriesExtremalEquations1997}.

In fact, for $n\neq 2$, $V^+$ is an \emph{irreducible} representation of $S\otimes C(n)^{\rm op}$ when $n$ is even, and decomposes into two (chiral) irreducible subrepresentations when $n$ is odd. (The second Clifford factor acts by right multiplication.)  This is in stark contrast to the structure of $V^+$ as an $\Fso(\eta)$-module, where it decomposes into an infinite number of irreducible subrepresentations. The latter is known as the Fischer decomposition.

\subsection{Summary of Results}
In Section \ref{sec:preliminaries}, we recall \cite{frappatDictionaryLieAlgebras2000a} the definition of the orthosymplectic Lie superalgebra $\mathfrak{osp}(1|2)$, the Weyl superalgebra $A(2n|n)$, and the oscillator representation $\mathfrak{osp}(1|2)\to A(2n|n)$ (discussed in \cite{williamsiiBasesInfinitedimensionalRepresentations2020, williamsii2024actionosp12npolytensor}), although we work in arbitrary metric. Of course, over the complex numbers, all Clifford algebras in given dimension are isomorphic, so our choices should be regarded as convenient if a real form were to be introduced at some stage.

In Section \ref{sec:Dirac-reduction-algebra}, we introduce the reduction algebra associated to the oscillator representation from Section \ref{sec:preliminaries}. This algebra appears in \cite[Section~7.5.3]{zhelobenkoRepresentationsReductiveLie1994} (see also \cite{zhelobenkoHypersymmetriesExtremalEquations1997}).
We recall the extremal projector \cite{tolstoyExtremalProjectorsContragredient2011}, \cite[Section~7.4.9]{zhelobenkoRepresentationsReductiveLie1994} for $\mathfrak{osp}(1|2)$ and use it to define the diamond product on a double coset space, following \cite{khoroshkinMickelssonAlgebrasZhelobenko2008,hartwigDiagonalReductionAlgebra2022}. This double coset space is a convenient realization of the reduction algebra.

In Section \ref{sec:relations}, we compute a presentation of the reduction algebra (in the form of the double coset space with diamond product) by generators and relations. Generators of the reduction algebra and commutation relations among them were given, with a sketch of proof, in \cite[Section~7.5.3]{zhelobenkoRepresentationsReductiveLie1994}; we give a complete presentation with full proofs.

In Section \ref{sec:operators}, we prove Theorem \ref{thm:product} which provides a formula for a product of $x$-generators. (One may call the $x$-generators \emph{raising operators}, as they increase the degree of a polynomial solution.) The resulting formula is used in Section \ref{sec:solutions} to prove Theorem \ref{thm:solution-acting-on-1} as we give the explicit degree $m$ polynomial solution to the Dirac equation provided by applying a sequence of raising operators. Indeed, we provide a description of the action of the generators of the reduction algebra on the space of solutions to the Dirac equation. 

Finally, in Section \ref{sec:irreducibility}, we show that, for $n\neq 2$, the constant solution $1$ is a generator when $V^+$ is considered as a module over the reduction algebra tensored with the Clifford algebra. That is, any polynomial solution to the Dirac equation can be obtained as a linear combination of solutions obtained from $1$ by applying raising operators on the left and/or multiplying by Clifford generators on the right. Compare \cite[Section~7.5.5]{zhelobenkoRepresentationsReductiveLie1994}, where every solution is obtained from constant solutions by applying raising operators. This leads to the statement that $V^+$ is either irreducible (when $n$ is even) or a direct sum of two irreducible modules (when $n$ is odd).

\section{Preliminaries} \label{sec:preliminaries}
We work over the complex field $\mathbb{C}$ unless another base ring is specified. Let $\mathbb{Z}_2$ stand for the two-element group $\mathbb{Z}/2\mathbb{Z} = \{\bar{0},\bar{1}\}$. A $\mathbb{Z}_2$-graded abelian group $M = M_{\bar{0}} \oplus M_{\bar{1}}$ carries a natural parity map on homogeneous elements: An element $x \in M_{\bar{i}}$ is homogeneous, $\bar{i} \in \mathbb{Z}_{2}$, in which case $|x| = \bar{i}$. We say $x$ is even when $|x| = \bar{0}$ and odd when $|x| = \bar{1}$. Necessarily, the zero element is the only element that is both even and odd. A homogeneous set $S$ is written as $\{x_{a}; y_{b}\}$ for $a \in I, b \in J$ with even (respectively, odd) elements written to the left (respectively, right) of the semicolon.

The category $\SVect$ of super vector spaces is the symmetric monoidal extension of the monoidal category of $\mathbb{Z}_2$-graded vector spaces $V = V_{\bar{0}} \oplus V_{\bar{1}}$ and parity-preserving morphisms by which the super braiding $\tau$ is given as 
\begin{equation}
   \tau\!: V \otimes W \cong W \otimes V, \quad \tau(a\otimes b) = (-1)^{|a||b|} b \otimes a, \label{equation:super-braiding}
\end{equation}
where use of $|\cdot|$ implies the argument is homogeneous and linearity is applied in other cases.
To be sure, $V \otimes W$ is a classical tensor product with $\mathbb{Z}_{2}$-grading given by $(V \otimes W)_{\bar{i}+\bar{j}} = (V_{\bar{i}} \otimes W_{\bar{j}}) \oplus (V_{\bar{i}+\bar{1}} \otimes W_{\bar{j}+\bar{1}})$ for $\bar{i}, \bar{j} \in \mathbb{Z}_{2}$.
\Cref{equation:super-braiding} dictates notions of forms, monoid objects, and Lie algebra objects internal to $\SVect$. Our base ring is then more aptly described as the purely even superring $\mathbb{C}^{1|0}$. The Lie algebra objects in $\SVect$ are called Lie superalgebras.

\begin{definition}[Lie superalgebra]\label{definition-Lie-superalgebra}
A super vector space $\mathfrak{g} = \mathfrak{g}_{\bar{0}}\oplus \mathfrak{g}_{\bar{1}}$ paired with a morphism $[\cdot,\cdot]\!:\mathfrak{g} \otimes \mathfrak{g} \rightarrow \mathfrak{g}$ is a Lie superalgebra if $[\cdot,\cdot]$ is skew-symmetric and satisfies the Jacobi identity. The morphism $[\cdot,\cdot]$ is then called the Lie superbracket of the Lie superalgebra $\mathfrak{g}$.
\end{definition}

\begin{remark}
Emphasizing the role of $\tau$, the equations defining skew-symmetry and the Jacobi identity in $\SVect$ are the same as the classical case except for use of the super braiding $\tau$ instead of the usual flip map $v \otimes w \mapsto w \otimes v$. Thus skew-symmetry of $[\cdot,\cdot]$ means \begin{equation*}[\cdot, \cdot] + [\cdot,\cdot] \circ \tau = 0,\end{equation*} and the Jacobi identity is 
\begin{equation*}[\cdot,\cdot] \circ (\Id \otimes [\cdot,\cdot]) \circ \sigma + [\cdot,\cdot] \circ (\Id \otimes [\cdot,\cdot]) \circ \sigma^{2} + [\cdot,\cdot] \circ (\Id \otimes [\cdot,\cdot]) \circ \sigma^{3} = 0, \quad \sigma = (\Id \otimes \tau) \circ (\tau \otimes \Id).\end{equation*} 
On the level of (homogeneous) elements, the following equations hold: 
\begin{align*}
[x,y] &= -(-1)^{|x||y|}[y,x]\\
[x,[y,z]] &= [[x,y],z] + (-1)^{|x||y|}[y,[x,z]].
\end{align*}
\end{remark}

The super vector space $\mathbb{C}^{1|2} = \mathbb{C} \oplus \mathbb{C}^{2}$ carries a canonical even nondegenerate symmetric bilinear form. The Lie superalgebra preserving the form is the orthosymplectic Lie superalgebra $\mathfrak{osp}(1|2)$.

\begin{definition}
    The Lie superalgebra $\mathfrak{osp}(1|2)$ is a super vector space with basis $\{h,e,f;x,y\}$ satisfying the following commutator (with respect to $\tau$) relations:
    \begin{align*}
        [h,e] &= 2e,&    [h,x] &= x,&  [h,y] &= -y,& [h,f] &= -2f,\\
        [y,y] &= -2f,&   [x,x] &= 2e,& [y,e] &= x,&  [x,f] &= y,\\
        [y,x] &= h,&        [e,f] &= h,& [y,f] &= 0,& [x,e] &= 0.
    \end{align*}
\end{definition}

\begin{definition}
Fix an invertible complex symmetric matrix $\eta$; let $\eta^{ij}$ denote the $i,j$-th entry of $\eta$.
The Weyl-Clifford Algebra $A(2n|n)$ of rank $n$ is the algebra $A(2n)\otimes_{\mathbb{C}} \operatorname{Cl}(\eta)$ with basis $\{x^1,\dotsc,x^n,\partial_1,\dotsc,\partial_n;\gamma^1,\dotsc,\gamma^n\}$ subject to the following relations:
\begin{equation}
\begin{aligned}
    x^ix^j - x^jx^i &= 0,\qquad & \partial_i\partial_j - \partial_j\partial_i &= 0,\qquad & \partial_ix^j - x^j\partial_i &= \delta_i^j\\
    \gamma^ix^j - x^j\gamma^i &= 0,&
    \gamma^i\partial_j - \partial_j\gamma^i &= 0, &
    \gamma^i\gamma^j + \gamma^j\gamma^i &= 2\eta^{ij}.
\end{aligned}
\end{equation}
\end{definition}

\begin{notation}
    Let $(\eta_{ij})_{i,j \in [n]}$ denote the inverse matrix of $(\eta^{ij})_{i,j \in [n]}$. 
    We put
    \begin{equation}
        x_i := \eta_{ij}x^j = \sum_{j=1}^{n}\eta_{ij}x^j,\qquad
        \partial^i := \eta^{ij}\partial_j = \sum_{j=1}^{n}\eta^{ij}\partial_j,\qquad
        \gamma_i := \eta_{ij}\gamma^j = \sum_{j=1}^{n}\eta_{ij}\gamma^j.
    \end{equation}
\end{notation}

We now introduce an embedding of $\mathfrak{osp}(1|2)$ into $A(2n|n)$ such that $x \in \mathfrak{osp}(1|2)$ is mapped to a scalar multiple of the Dirac operator.

\begin{prop}\label{prp:embedding}
The mapping $\psi : \mathfrak{osp}(1\vert 2) \to A(2n|n)$ given by
\begin{subequations}
\begin{align}
X := \psi(x) &= \frac{i}{\sqrt{2}}\gamma^i\partial_i\\
Y := \psi(y) &= \frac{i}{\sqrt{2}}\gamma^ix_i\\
H := \psi(h) &= -\frac{1}{2}(\partial_ix^i + x^i\partial_i)\\
E := \psi(x)^2 = \psi(e) &= -\frac{1}{2}\partial^i\partial_i\\
F := -\psi(y)^2 = \psi(f) &= \frac{1}{2}x^ix_i
\end{align}
\end{subequations}
is an injective super Lie algebra homomorphism.
\end{prop} 
\begin{proof}
    This is a slight modification of the oscillator representation as in \cite[for example]{frappatDictionaryLieAlgebras2000a}; compare also \cite[Section~7.5.2]{zhelobenkoRepresentationsReductiveLie1994}. We provide a proof in Appendix \ref{appendix:embedding}.
\end{proof}

\section{Dirac Reduction Algebra}
\label{sec:Dirac-reduction-algebra}

\begin{definition}
    Let $\mathcal{A}=A(2n|n)$ and let  $\mathfrak{g} \subset \mathcal{A}$ denote the image of $\mathfrak{osp}(1|2)$ as described in Proposition \ref{prp:embedding}. Since $\mathfrak{g}$ is isomorphic to $\mathfrak{osp}(1|2)$, it admits a triangular decomposition $\mathfrak{g} = \mathfrak{g}_{-}\oplus \mathfrak{h} \oplus \mathfrak{g}_{+}$ given by
    \begin{equation}
        \mathfrak{g}_{+} := \mathbb{C}X \oplus \mathbb{C}E,\qquad 
        \mathfrak{h} := \mathbb{C}H,\qquad 
        \mathfrak{g}_{-} := \mathbb{C}Y \oplus \mathbb{C}F.
    \end{equation}
    We define $D$ as the multiplicative set generated by
    \begin{align}
        \{ H + k\cdot 1 : k \in \mathbb{Z}\}.
    \end{align} 
    Let $\mathcal{A}'$ denote the localization of $\mathcal{A}$ at $D$. Let $I' := \mathcal{A}'\mathfrak{g}_{+}$.
\end{definition}

\begin{prop}\label{prp:2.1}
    The following relations holds in the algebra $\mathcal{A}'$:
\begin{align*}
    [X,x^k] &= \frac{i}{\sqrt{2}}\gamma^k,&
    [X,\partial_k] &= 0,&
    [X,\gamma^k] &= 2\frac{i}{\sqrt{2}}\partial^k\\
    [Y,x^k] &= 0,&
    [Y,\partial_k] &= -\frac{i}{\sqrt{2}}\gamma_k,&
    [Y,\gamma^k] &= 2\frac{i}{\sqrt{2}}x^k\\
    [H, x^k] &= -x^k,&
    [H, \partial_k] &= \partial_k,&
    [H,\gamma^k] &= 0.
\end{align*}
\end{prop} 

\begin{proof}
    See Appendix \ref{appendix:adjoint-actions}.
\end{proof}

\begin{definition}
    An $\mathcal{A}'$-module $V$ is called \emph{locally $\mathfrak{g}_{+}$-finite} if for every $v\in V$, we have $\operatorname{dim}_{\mathbb{C}}U(\mathfrak{g}_{+})v < \infty$. 
\end{definition}

\begin{prop}\label{Proposition:locally-finite-module}
Let $M' := \mathcal{A}'/I'$. Then $M'$ is a locally $\mathfrak{g}_{+}$-finite module. 
\end{prop}

\begin{proof}
We prove the proposition by induction on the word length $m$ of monomials of $\{x^j,\partial_j,\gamma^j : j \in [n]\}$.
The base case $m = 0$ holds as the empty word written as $1$ is a highest weight vector. For $m = 1$, we have that
\begin{align*}
    \operatorname{ad}(X)^3(a_1) = 0
\end{align*}
holds for any $a_1 \in \{x^1,\dotsc,x^n,\partial_1,\dotsc,\partial_n,\gamma^1,\dotsc,\dotsc,\gamma^n\}$.
By induction and the super-Leibniz rule, and considering relation of positive root vectors in $\mathfrak{osp}(1|2)$, we find that there exists some $N \geq 1$ so that the power $\operatorname{ad}(X)^{N}$ annihilates any word, and hence any element in $M'$.
\end{proof}

\begin{definition}
Suppose $V$ is a locally $\mathfrak{g}_{+}$-finite $\mathcal{A}'$-module.
Then the super vector space of $\mathfrak{g}_{+}$-invariants of $V$ is defined as 
    \begin{align}
        V^{+} := \{v \in V : \mathfrak{g}_{+}v = 0\}.
    \end{align}
    Dually, the super vector space of $\mathfrak{g}_{-}$-coinvariants is defined as
    \begin{align}
        V_{-} := V/\mathfrak{g}_{-}V.
    \end{align}
    Following \cite{hartwigDiagonalReductionAlgebra2022}, we say that an \textit{extremal projector} exists for a Lie superalgebra $\mathfrak{g}$ if for all locally $\mathfrak{g}_{+}$-finite $\mathcal{A}'$-modules $V$, there exists a superlinear map $P_V :V \to V$ at $V$  satisfying
    \begin{enumerate}[{\rm (i)}]
        \item $f\circ P_V = P_W \circ f$ \text{ for all $\mathcal{A}'$-module maps $f : V \to W$},
        \item $P_V(v) = v$ \text{ for all $v \in V^{+}$},
        \item $P_V(v) + \mathfrak{g}_{-}V = v + \mathfrak{g}_{-}V$,
        \item $\operatorname{im}P_V \subset V^{+}$
        \item $\mathfrak{g}_{-}V \subset \operatorname{ker}P_V$
    \end{enumerate}
    $P_V$ is called the \emph{extremal projector at $V$}.
\end{definition}

Tolstoy \cite{tolstoyExtremalProjectorsContragredient2011} and collaborators have shown that all basic classical Lie superalgebras (including $\mathfrak{osp}(1|2)$) has an extremal projector. We recall the details for $\mathfrak{osp}(1|2)$ in the following theorem.

\begin{thm}
Let $T\mathcal{A}'$ denote the Taylor extension of $\mathcal{A}'$ given by the projective limit
\begin{align}
    T\mathcal{A}' := \varprojlim \frac{\mathcal{A}'}{Y^n\mathcal{A}' + \mathcal{A}'X^n}.
\end{align}
For any locally $\mathfrak{g}_{+}$-finite $\mathcal{A}'$-module $V$, left-action on $V$ by
\begin{align}
    P = \sum_{n=0}^{\infty}\varphi_n(H)Y^nX^n \in T\mathcal{A}',
\end{align}
with polynomial coefficients given by
\begin{align}
    \varphi_n(H) = \prod_{k=1}^{n}\frac{(-1)^{k}}{\kappa_k(H)},
\end{align}
where
\begin{align}
    \kappa_n(H) := \begin{cases}
        \frac{-n}{2} &\text{ if $n$ is even}\\
        H + \frac{n+1}{2} &\text{ if $n$ is odd,}
    \end{cases}
\end{align}
is a well-defined map of vector super spaces. Furthermore, $P_V : V \to V$ given by $P_V(v) = P\cdot v$ is the extremal projector at $V$.
\end{thm}
\begin{proof}
    Write \begin{align*}
        P = \sum_{n=0}^{\infty}\varphi_n(H)Y^{n}X^{n},
    \end{align*}
    where $\varphi_0(H) = 1$, and $\varphi_n(t) \in \mathbb{C}[t]$ are polynomials to be determined.
    Note by the Leibniz rule, we have for $n \geq 1$
    \begin{align*}
        [X,Y^{n}] &= \sum_{k=1}^{n}(-1)^{k-1}Y^{k-1}HY^{n-k}\\
        &= \sum_{k=1}^{n}(-1)^{k-1}(H+k-1)Y^{n-1}\\
        &= \bigg(\bigg(\sum_{k=1}^{n}(-1)^{k-1}\bigg)(H-1) + \bigg(\sum_{k=1}^{n}(-1)^{k-1}k\bigg)\bigg)Y^{n-1}\\
        &= \bigg(\bigg(\frac{1-(-1)^{n}}{2}\bigg)(H-1) + \bigg(\frac{1-(2n+1)(-1)^{n}}{4}\bigg)\bigg)Y^{n-1}\\
        &= \kappa_n(H-1)Y^{n-1}.
    \end{align*}
    As such,
    \begin{align*}
        XP &= \sum_{n=0}^{\infty}X\varphi_n(H)Y^{n}X^{n}\\
        &=\sum_{n=0}^{\infty}\varphi_n(H-1)XY^{n}X^{n}\\
        &=\sum_{n=0}^{\infty}\varphi_n(H-1)((-1)^{n}Y^{n}X + [X,Y^{n}])X^{n}\\
        &=\sum_{n=0}^{\infty}\varphi_n(H-1)((-1)^{n}Y^{n}X + \kappa_n(H-1)Y^{n-1})X^{n}\\
        &= \sum_{n=0}^{\infty}\varphi_n(H-1)(-1)^{n}Y^{n}X^{n+1} + \sum_{n=1}^{\infty}\varphi_n(H-1)\kappa_n(H-1)Y^{n-1}X^{n}\\
        &= \sum_{n=0}^{\infty}\varphi_n(H-1)(-1)^{n}Y^{n}X^{n+1} + \sum_{n=0}^{\infty}\varphi_{n+1}(H-1)\kappa_{n+1}(H-1)Y^{n}X^{n+1}\\
        &= \sum_{n=0}^{\infty}\bigg(\varphi_n(H-1)(-1)^{n} + \varphi_{n+1}(H-1)\kappa_{n+1}(H-1)\bigg)Y^{n}X^{n+1}.
    \end{align*}
    As such, $XP = 0$ if and only if
    \begin{align*}
        \varphi_{n+1}(H) = \frac{(-1)^{n+1}}{\kappa_{n+1}(H)}\varphi_{n}(H).
    \end{align*}
    Thus, each $\varphi_n(H)$ is uniquely determined by $\varphi_0(H) = 1$, with
    \begin{align*}
        \varphi_n(H) = \prod_{k=1}^{n}\frac{(-1)^{k}}{\kappa_k(H)}.
    \end{align*}
    The first few terms are $\varphi_0(H)=1$ and
    \begin{equation}\label{eq:first-phis}
    \begin{aligned}
        \varphi_1(H) &= \frac{-1}{H+1}, &\varphi_2(H) &= \frac{1}{H+1},\\ \varphi_3(H) &= \frac{-1}{(H+1)(H+2)},& \varphi_4(H) &= \frac{1}{2(H+1)(H+2)}.
    \end{aligned}
    \end{equation}
    In fact, for any $k \geq 0$, one can verify by induction that
    \begin{align}
        \varphi_{2k+1}(H) = \frac{-1}{k!}\prod_{m=0}^{k}\frac{1}{H+m+1},\label{eq:phi-odd}\\
        \varphi_{2k}(H) = \frac{1}{k!}\prod_{m=0}^{k-1}\frac{1}{H+m+1}.\label{eq:phi-even}
    \end{align}
\end{proof}

\begin{definition}
    The normalizer of a left ideal $J \unlhd R$ is defined as
    \begin{align}
        N_{R}(J) := \{x \in R : Jx \subset J\}.
    \end{align} 
    With this, the D-localized reduction algebra of $\mathfrak{g}$ in $\mathcal{A}$ is defined as
    \begin{align}
        Z_n' := N_{\mathcal{A}'}(I')/I'.
    \end{align}
    The reduction algebra of $\mathfrak{g}$ in $\mathcal{A}$ is defined as
    \begin{align*}
        Z_n := N_{\mathcal{A}}(I)/I.
    \end{align*}
\end{definition}

It follows that $Z_n' = (M')^{+}$. Additionally, $(M')_{-}$ can be identified as the double coset space $\mathcal{A}'/\II$: where $\II := \mathfrak{g}_{-}\mathcal{A}' + \mathcal{A}'\mathfrak{g}_{+}$. This  identification is given by $(v + \mathcal{A}'\mathfrak{g}_{+}) + \mathfrak{g}_{-}M' \mapsto v + \II$. 

By the properties of the extremal projector, it follows that the induced map $\textnormal{\textbf{P}}_M : M_{-} \to M^{+}$ is an isomorphism of vector spaces. This isomorphism allows us to define a product $\bDiamond$ on $\mathcal{A}'/\II$ that requires $\textnormal{\textbf{P}}_M$ to be an isomorphism of algebras between D-localized reduction algebra $Z_n'$ and $(\mathcal{A}'/\II,\bDiamond)$.

\begin{definition}
    Let $\overline{a},\overline{b} \in \mathcal{A}'/\II$. Since $X$ acts locally finitely on $M'$, it follows that $X$ acts locally finitely on $\mathcal{A}'/\II$ as well. The diamond product $\bDiamond$ on $\mathcal{A}'/\II$ is then defined as
    \begin{align}
        \overline{a}\bDiamond \overline{b} := aPb + \II.
    \end{align}
\end{definition}

\section{Relations} \label{sec:relations}   

The following result, where the diamond products of the generators of $A(2n|n)$ are computed, will be used to deduce a complete presentation of the reduction algebra in Theorem \ref{thm:presentation}.

\begin{prop}
    For any $a \in \mathcal{A}'$, write $\overline{a} := a + \II$. Then the following equalities hold for all $i,j \in [n]$ and $\phi(t) \in \mathbb{C}(t)$,
    \begin{subequations}
    \begin{align}
    \overline{a}\bDiamond\overline{\partial_j} &= a\partial_j + \II\\
    \overline{x^i}\bDiamond\overline{a} &= x^ia + \II\\
    \overline{\partial_i} \bDiamond \overline{\gamma^j} &= \partial_i\gamma^j - \varphi_1(H-1)\gamma_i\partial^j + \II\\
    \overline{\gamma^i} \bDiamond \overline{\gamma^j} &= \gamma^i\gamma^j - 2\varphi_1(H)x^i\partial^j + \II\\
    \overline{\partial_i} \bDiamond \overline{x^j} &= \partial_ix^j - \frac{\varphi_1(H-1)}{2}\gamma_i\gamma^j + \varphi_2(H-1)x_i\partial^j + \II\\
    \overline{\gamma^i} \bDiamond \overline{x^j} &= \gamma^ix^j - \varphi_1(H)x^i\gamma^j + \II\\
    \overline{\phi(H)}\bDiamond\overline{a} &= \phi(H)a + \II\\
    \overline{a}\bDiamond\overline{\phi(H)} &= a\phi(H) + \II.
    \end{align}
    \end{subequations}
\end{prop}

\begin{proof}
    Let $a,b \in \mathcal{A}'$. Let $N \geq 1$ be such that $X^Na \in \II$. Then
    \begin{align*}
        \overline{a}\bDiamond \overline{b} &= ab + \sum_{n=1}^{N-1}a\varphi_n(H)Y^nX^nb + \II\\
        &= ab + \sum_{n=1}^{N-1}aY^n\varphi_n(H-n)X^nb + \II.
    \end{align*}
    Before proceeding, we will now present the following descriptions of left and right actions by $X$ and $Y$, respectively. For any $c \in \mathcal{A}'$,
    \begin{align}
        &cX^nb = c\ad(X)^n(b)\\
        &aY^nc = \ad_r(Y)^n(a)c,
    \end{align}
    where $\ad_r(Y)(x) := [x,Y]$ for any $x \in \mathcal{A}'$.
    To see the above, first note that for any $c \in \mathcal{A}'$,
    \begin{align*}
        cXb + \II &= cXb - c(-1)^{|X||b|}bX + \II\\
        &= c\ad(X)(b) + \II.
    \end{align*}
    By induction,
    \begin{align*}
        cX^nb + \II = cX^{n-1}\ad(X)(b) + \II = c\ad(X)^{n-1}(\ad(X)(b)) + \II = c\ad(X)^{n}(b) + \II.
    \end{align*}
    Similarly, for any $c \in \mathcal{A}'$,
    \begin{align*}
        aYc + \II = aYc - (-1)^{|a||Y|}Yac + \II = [a,Y]c + \II = \ad_r(Y)(a)c + \II.
    \end{align*}
    By induction,
    \begin{align*}
        aY^nc + \II &= \ad_r(Y)(a)Y^{n-1}c + \II = \ad_r(Y)^{n-1}(\ad_r(Y)(a))c + \II = \ad_r(Y)^{n}(a)c + \II.
    \end{align*}
    As such,
    \begin{align}
        \overline{a}\bDiamond \overline{b} = ab + \sum_{n=1}^{N-1}\ad_r(Y)^n(a)\varphi_n(H-n)\ad(X)^nb + \II.
    \end{align}

   Let $i,j \in [n]$. Since $\ad(X)(\partial_i) = 0$, we have $\overline{a}\bDiamond\overline{\partial_j} = a\partial_j + \II$. Likewise, since $\ad_r(Y)(x^i) = 0$, we have $\overline{x_i}\bDiamond\overline{a} = x_ia + \II$.

   Now let $\phi(t) \in \mathbb{C}(t)$ be arbitrary. Then $X\phi(H) = \phi(H-1)X \in \II$. Hence $\overline{a}\bDiamond\overline{\phi(H)} = a\phi(H)$.
    Likewise, $\phi(H)Y = Y\phi(H-1) \in \II$. Hence $\overline{\phi(H)}\bDiamond\overline{a} = \phi(H)a$.

    Next, using Proposition \ref{prp:2.1} and that $\varphi(H+1)x^k = x^k\varphi(H)$ and $\varphi(H-1)\partial_k = \partial_k\varphi(H)$, we compute:
    \begin{align*}
        \overline{\partial_i} \bDiamond \overline{\gamma^j} &= \partial_i\gamma^j + \ad_r(Y)(\partial_i)\varphi_1(H-1)\ad(X)(\gamma^j) + \II \\
        &= \partial_i\gamma^j + (\frac{i}{\sqrt{2}}\gamma_i)\varphi_1(H-1)(2\cdot\frac{i}{\sqrt{2}}\partial^j) + \II \\
        &= \partial_i\gamma^j - 2\cdot\frac{1}{2}\gamma_i\varphi_1(H-1)\partial^j + \II \\
        &=\partial_i\gamma^j - \varphi_1(H-1)\gamma_i\partial^j + \II,
        \end{align*}
        
    \begin{align*}
    \overline{\gamma^i} \bDiamond \overline{\gamma^j} &= \gamma^i\gamma^j + \ad_r(Y)(\gamma^i)\varphi_1(H-1)\ad(X)(\gamma^j) + \II \\
    &=\gamma^i\gamma^j + (2\frac{i}{\sqrt{2}}x^i)\varphi_1(H-1)(2\frac{i}{\sqrt{2}}\partial^j) + \II \\
    &= \gamma^i\gamma^j - 2x^i\varphi_1(H-1)\partial^j + \II \\
    &= \gamma^i\gamma^j - 2\varphi_1(H)x^i\partial^j + \II,
    \end{align*}
    
    \begin{align*}
    \overline{\partial_i} \bDiamond \overline{x^j} &= \partial_ix^j + \ad_r(Y)(\partial_i)\varphi_1(H-1)\ad(X)(x^j) + \ad_r(Y)^2(\partial_i)\varphi_2(H-2)\ad(X)^2(x^j) + \II \\
    &= \partial_ix^j + (\frac{i}{\sqrt{2}}\gamma_i)\varphi_1(H-1)(\frac{i}{\sqrt{2}}\gamma^j) + (-x_i)\varphi_2(H-2)(-\partial^j) + \II \\
    &= \partial_ix^j - \frac{1}{2}(\gamma_i)\varphi_1(H-1)(\gamma^j) + (x_i)\varphi_2(H-2)(\partial^j) + \II \\
    &= \partial_ix^j - \frac{\varphi_1(H-1)}{2}\gamma_i\gamma^j + \varphi_2(H-1)x_i\partial^j + \II,
    \end{align*}
    and
    \begin{align*}
    \overline{\gamma^i} \bDiamond \overline{x^j} &=\gamma^ix^j + \ad_r(Y)(\gamma^i)\varphi_1(H-1)\ad(X)(x^j) + \ad_r(Y)^2(\gamma^i)\varphi_2(H-2)\ad(X)^2(x^j) + \II \\
    &= \gamma^ix^j + (2\frac{i}{\sqrt{2}}x^i)\varphi_1(H-1)(\frac{i}{\sqrt{2}}\gamma^j) + (0)\varphi_2(H-2)(\frac{-1}{2}\partial^j) + \II \\
    &=\gamma^ix^j - \varphi_1(H)x^i\gamma^j + \II.
    \end{align*}
\end{proof}

We are now ready to state and prove the first main theorem of the paper. Generators of $Z_n$ and commutation relations among them appear, with a sketch of proof, in \cite[Section~7.5.3]{zhelobenkoRepresentationsReductiveLie1994}.

\begin{thm}\label{thm:presentation}
    The following relations hold for all $i,j \in [n]$, $\phi(t) \in \mathbb{C}(t)$:
    \begin{subequations}
    \begin{align}
    \overline{\partial_i} \bDiamond \overline{\gamma^j} - \overline{\gamma^j}\bDiamond \overline{\partial_i} &= \frac{1}{H}\overline{\gamma_i}\bDiamond \overline{\partial^j}\\
    \overline{\gamma^i} \bDiamond \overline{\gamma^j} + \overline{\gamma^j}\bDiamond\overline{\gamma^i}&= 2\eta^{ij}\overline{1} + \frac{2}{H+1}\bigg(\overline{x^j}\bDiamond\overline{\partial^i} + \overline{x^i}\bDiamond\overline{\partial^j}\bigg)\\
    \overline{\partial_i} \bDiamond \overline{x^j} - \overline{x^j}\bDiamond\overline{\partial_i} &= \delta_{i}^j\overline{1} + \frac{1}{2H}\overline{\gamma_i}\bDiamond\overline{\gamma^j} + \frac{1}{H+1}\overline{x_i}\bDiamond\overline{\partial^j}\\
    \overline{\gamma^i} \bDiamond \overline{x^j} - \overline{x^j}\bDiamond\overline{\gamma^i}&= \frac{1}{H+1}\overline{x^i}\bDiamond\overline{\gamma^j}\\
    \overline{x^i}\diamond\overline{\partial_i} &= (-\frac{n}{2}-H)\bar 1\\
    \overline{\gamma^i}\diamond\overline{\partial_i}&=0\\
    \overline{\gamma_i}\diamond\overline{x^i}&=0
    \end{align}
    \end{subequations}
\end{thm}

\begin{proof}
We use the fact that $\bDiamond$ is associative. Additionally, since $\overline{H}\bDiamond \overline{a} = Ha + \II$, we will simply write $H\overline{a} := Ha + \II$.
    \begin{align*}
    \overline{\partial_i} \bDiamond \overline{\gamma^j} &= \partial_i\gamma^j - \varphi_1(H-1)\gamma_i\partial^j + \II\\
    &= \gamma^j\partial_i - \varphi_1(H-1)\gamma_i\partial^j + \II\\
    &= \overline{\gamma^j}\bDiamond\overline{\partial_i} - \varphi_1(H-1)\overline{\gamma_i}\bDiamond\overline{\partial^j} + \II\\
    &= \overline{\gamma^j}\bDiamond\overline{\partial_i} + \frac{1}{H}\overline{\gamma_i}\bDiamond\overline{\partial^j} + \II,
    \end{align*}

    \begin{align*}
    \overline{\gamma^i} \bDiamond \overline{\gamma^j} &= \gamma^i\gamma^j - 2\varphi_1(H)x^i\partial^j + \II\\
    &= \gamma^i\gamma^j + \frac{2}{H+1}\overline{x^i}\bDiamond\overline{\partial^j} + \II,
    \end{align*}

    \begin{align*}
        \gamma^i\gamma^j = \overline{\gamma^i} \bDiamond \overline{\gamma^j} - \frac{2}{H+1}\overline{x^i}\bDiamond\overline{\partial^j} + \II,
    \end{align*}
    
    \begin{align*}
    \overline{\gamma^i}\bDiamond\overline{\gamma^j} &= \gamma^i\gamma^j + \frac{2}{H+1}\overline{x^i}\bDiamond\overline{\partial^j} + \II\\
    &=  (2\eta^{ij} - \gamma^j\gamma^i + \II) + \frac{2}{H+1}\overline{x^i}\bDiamond\overline{\partial^j}\\
    &=  \overline{2\eta^{ij}} - (\overline{\gamma^j} \bDiamond \overline{\gamma^i} - \frac{2}{H+1}\overline{x^j}\bDiamond\overline{\partial^i}) + \frac{2}{H+1}\overline{x^i}\bDiamond\overline{\partial^j}\\
    &= \overline{2\eta^{ij}} - \overline{\gamma^i}\bDiamond\overline{\gamma^j} + \frac{2}{H+1}\bigg(\overline{x^j}\bDiamond\overline{\partial^i} + \overline{x^i}\bDiamond\overline{\partial^j}\bigg),
    \end{align*}

    \begin{align*}
    \overline{\partial_i} \bDiamond \overline{x^j} &= \partial_ix^j - \frac{\varphi_1(H-1)}{2}\gamma_i\gamma^j + \varphi_2(H-1)x_i\partial^j + \II\\
    &= \delta_i^j + x^j\partial_i + \frac{1}{2H}(\overline{\gamma_i}\bDiamond\overline{\gamma^j} - \frac{2}{H+1}\overline{x_i}\bDiamond\overline{\partial^j}) + \frac{1}{H}\overline{x_i}\bDiamond\overline{\partial^j} + \II\\
    &= \overline{\delta_{i}^{j}} + \overline{x^j}\bDiamond\overline{\partial_i} + \frac{1}{2H}\overline{\gamma_i}\bDiamond\overline{\gamma^j} - \frac{1}{H(H+1)}\overline{x_i}\bDiamond\overline{\partial^j} + \frac{1}{H}\overline{x_i}\bDiamond\overline{\partial^j}\\
    &=\overline{\delta_{i}^{j}} + \overline{x^j}\bDiamond\overline{\partial_i} + \frac{1}{2H}\overline{\gamma_i}\bDiamond\overline{\gamma^j} + (\frac{-1}{H(H+1)} + \frac{H+1}{H(H+1)})\overline{x_i}\bDiamond\overline{\partial^j}\\
    &=\overline{\delta_{i}^{j}} + \overline{x^j}\bDiamond\overline{\partial_i} + \frac{1}{2H}\overline{\gamma_i}\bDiamond\overline{\gamma^j} + \frac{1}{H+1}\overline{x_i}\bDiamond\overline{\partial^j},
    \end{align*}
and

    \begin{align*}
    \overline{\gamma^i} \bDiamond \overline{x^j} &= \gamma^ix^j - \varphi_1(H)x^i\gamma^j + \II\\
    &= \overline{x^j}\bDiamond\overline{\gamma^i} + \frac{1}{H+1}\overline{x^i}\bDiamond\overline{\gamma^j}.
    \end{align*}
\end{proof}

\section{Operators}\label{sec:operators}

We return to the Weyl-Clifford superalgebra $\mathcal{A}$. Note that there is a natural action of this algebra on the space $V$ of Clifford-valued polynomials given by
\begin{equation}
    x^i\cdot \varphi = x^i\varphi,\qquad 
    \gamma^{i}\cdot \varphi = \gamma^i\varphi,\qquad
    \partial_i\cdot \varphi = \frac{\partial}{\partial x^i}\varphi.
\end{equation}

When we restrict this action to the space of solutions to the Dirac equation, denoted by $V^{+}$, we have that $I = \mathcal{A}X$ acts by zero. As such, we obtain a well-defined action of the reduction algebra $Z_n$ on the space of solutions to the Dirac equation. In this section, we seek a general form for the solutions we obtain through the application of the generators of the localized reduction algebra normalized to become operators in the ordinary reduction algebra.

To this end, note that the localized reduction algebra cannot necessarily act on the space of solutions. Indeed, consider the following operator:
\begin{align*}
    (H + I) \cdot 1 = \frac{-n}{2}.
\end{align*}
Whenever $n$ is even, $H + I$ acts by a scalar. Hence the action of $1/(H + \frac{n}{2}) + I'$ is not well-defined.  
To avoid this, we introduce operators multiplied by polynomials of $H$ to clear the denominators.

\begin{align}
    \widehat{x^i} &:= (H+1)P(x^i + \II) = (H+1)x^i + \frac{1}{2}\gamma^jx_j\gamma^i + \frac{1}{2}x^jx_j\partial^i + I' \in Z_n'\\
    \widehat{\gamma^i} &:= (H+1)P(\gamma^i + \II) = (H+1)\gamma^i + \gamma^jx_j\partial^i + I' \in Z_n'\\
    \widehat{\partial_i} &:= \partial_i + I' \in Z_n'.
\end{align}

Note that the above operators actually belong to the image of $Z_n = N(I)/I$ in the localization.

We wish to find a closed formula for the solutions that result when we act by repeated application of the $\widehat{x^{i}}$. First, we make the following computation.

\begin{align*}
\widehat{x^{i_1}}\cdots \widehat{x^{i_m}} &= P\big((H+1)\bar x^{i_1}\diamond\cdots \diamond(H+1)\bar x^{i_m}\big) + I' \\ 
&= (H+1)(H+2)\cdots (H+m)P(\overline{x^{i_1}\cdots x^{i_m}}) + I'\\
&= (H+1)(H+2)\cdots (H+m)P(x^{i_1}\cdots x^{i_m}) + I'\\
&=(H+1)(H+2)\cdots (H+m)\sum_{k=0}^\infty\varphi_k(H)Y^k (\ad X)^k ( x^{i_1}\cdots x^{i_m}) + I'.
\end{align*}

We introduce the following notation to make computations easier with the above formula.\\

Let $a,b,c,d$ be integers such that $2a + b + c + d = m$ and fix integers $i_1,\dotsc,i_m \in [n]$. Then let $S$ denote the set of $m$-tuples containing all the integers $1,\dotsc,m$.

For an $m$-tuple $\alpha$, write 

\begin{align*}
    (\eta^ax^r\gamma^s\partial^t)^{i_\alpha} := \eta^{i_{\alpha_1}i_{\alpha_2}}\dotsb \eta^{i_{\alpha_{2a-1}}i_{\alpha_{2a}}}x^{i_{\alpha_{2a+1}}}\dotsb x^{i_{\alpha_{2a+r}}}\gamma^{i_{\alpha_{2a+r+1}}} \\ \quad \dotsb\gamma^{i_{\alpha_{2a+r+s}}}\partial^{i_{\alpha_{2a+r+s+1}}}\dotsb\partial^{i_{\alpha_{2a+r+s+t}}}.
\end{align*}

Then consider orbit equivalence $\sim$ on $S$ under action by a subgroup $G_{a,r,s,t} \cong ((S^2)^{a} \rtimes S^a) \times S^{r} \times (S^1)^{s} \times S^t$ of $S^m$ with action given by letting $(S^2)^{a}\rtimes S^a$ act on the first $2a$ elements, letting $S^r$ act on the next $r$ elements, and letting $S^t$ act on the last $t$ elements of the $m$-tuple. The action of $S^r$ and $S^t$ on their respective sets of elements is the standard permutation action. The action of $(S^2)^{a}\rtimes S^a$ is such that the generator of the $i$-th copy of $S^2$ permutes the $(2i-1)$-th and $2i$-th elements of the $m$-tuple. The semidirect product of $S^a$ permutes the $a$ groups of two elements.

Explicitly, the action is given as follows:
\begin{align*}
    &(\tau_1,\dotsc,\tau_a,\sigma_a,\sigma_r,\sigma_t) \cdot (k_1,\dotsc,k_m) \\
    &= (k_{2*\sigma_a(1)-\tau_1(1)},k_{2*\sigma_a(1)-\tau_1(0)},k_{2*\sigma_a(2) - \tau_2(1)},k_{2*\sigma_a(2) - \tau_2(0)},\dotsc,k_{2*\sigma_a(a)-\tau_a(1)},k_{2*\sigma_a(a) - \tau_a(0)},\\
    &k_{2a+\sigma_r(1)},\dotsc,k_{2a+\sigma_r(r)},k_{2a+r+1},\dotsc,k_{2a+r+s},k_{2a+r+s+\sigma_t(1)},\dotsc,k_{2a+r+s+\sigma_t(t)}).
\end{align*}

The group acts faithfully on the set of $m$-tuples of integers from $1,\dotsc,m$. As such, it can be identified as a subgroup of $S^m$. It follows that whenever two $m$-tuples lie in the same orbit, that one can be obtained by permuting the order of indices on the $\eta$'s, the order of the $\eta$'s themselves, the order of the $x$'s, or the order of the $\partial$'s. In other words, $(\eta^a x^r \gamma^s \partial^t)^{i_{\alpha}} = (\eta^a x^r \gamma^s \partial^t)^{i_{\alpha'}}$ whenever $\alpha \sim \alpha'$. This leads us to define the following symbol:

\begin{equation}\label{eq:bracket-def}
    [\eta^a x^r \gamma^s \partial^t ; i_1,\dotsc,i_m] := \sum_{[\alpha] \in S/\sim}(\eta^a x^r \gamma^s \partial^t)^{i_\alpha}.
\end{equation}

We will also allow $r$ to be negative. In that case, define $[\eta^a x^r \gamma^s \partial^t ; i_1,\dotsc,i_m] := 0$.

\begin{lemma}
For fixed $i_1,\dotsc,i_m \in [n]$, we have the following:
    \begin{equation}
        [\eta^{a}x^{r}\gamma^{s+2}\partial^t ; i_1,\dotsc,i_m] = 2(a+1)[\eta^{a+1}x^{r}\gamma^{s}\partial^t; i_1,\dotsc,i_m].
    \end{equation}
\end{lemma}

\begin{proof}
First, we identify $G_{a,r,s,t} := ((S^2)^a \rtimes S^a) \times S^r \times (S^1)^{s+2} \times S^t$ as a subgroup of $S^m$. Then we can write each $m$-tuple as a permutation of $\alpha := (1,2,\dotsc,m)$. We can hence write
\begin{align*}
    [\eta^a x^r \gamma^{s+2} \partial^t ; i_1,\dotsc,i_m] &= \sum_{[\sigma] \in S^m/G_{a,r,s,t}}(\eta^a x^r \gamma^{s+2} \partial^t)^{i_{\sigma\cdot \alpha}}\\
    &= \sum_{[\sigma] \in S^m/G_{a,r,s,t}}\frac{1}{2^aa!r!t!}\sum_{\sigma' \in [\sigma]}(\eta^a x^r \gamma^{s+2} \partial^t)^{i_{\sigma'\cdot \alpha}}\\
    &= \frac{1}{2^aa!r!t!}\sum_{\sigma \in S^m}(\eta^a x^r \gamma^{s+2} \partial^t)^{i_{\sigma\cdot \alpha}}.
\end{align*}
Then 
\begin{align*}
    &\sum_{\sigma \in S^m}(\eta^a x^r \gamma^{s+2} \partial^t)^{i_{\sigma\cdot \alpha}} \\
    &= \sum_{\sigma \in S^m}(\cdot\cdot\cdot)\gamma^{i_{\sigma(2a+r+1)}}\gamma^{i_{\sigma(2a+r+2)}}(\cdot\cdot\cdot)\\
    &= \frac{1}{2}\sum_{\sigma \in S^m}(\cdot\cdot\cdot)(\gamma^{i_{\sigma(2a+r+1)}}\gamma^{i_{\sigma(2a+r+2)}} + \gamma^{i_{\sigma(2a+r+2)}}\gamma^{i_{\sigma(2a+r+1)}})(\cdot\cdot\cdot)\\
    &= \sum_{\sigma \in S^m}(\cdot\cdot\cdot)(\eta^{i_{\sigma(2a+r+1)}i_{\sigma(2a+r+2)}})(\cdot\cdot\cdot).
\end{align*}
Using a permutation switching $2a+r+1$ with $2a+1$ and $2a+r+2$ with $2a+2$, we have that 
\begin{align*}
    \sum_{\sigma \in S^m}(\eta^a x^r \gamma^{s+2} \partial^t)^{i_{\sigma\cdot \alpha}} = \sum_{\sigma \in S^m}(\eta^{a+1}x^{r}\gamma^{s}\partial^{t})^{i_{\sigma\cdot \alpha}}.
\end{align*}
Hence
\begin{align*}
    [\eta^ax^r\gamma^{s+2}\partial^t;i_1,\dotsc,i_m] &= \frac{1}{2^{a}a!r!t!}\sum_{\sigma \in S^m}(\eta^{a+1}x^{r}\gamma^{s}\partial^{t})^{i_{\sigma\cdot \alpha}}\\
    &= \frac{2(a+1)}{2^{a+1}(a+1)!r!t!}\sum_{\sigma \in S^m}(\eta^{a+1}x^{r}\gamma^{s}\partial^{t})^{i_{\sigma\cdot \alpha}} = 2(a+1)[\eta^{a+1}x^r\gamma^s\partial^t;i_1,\dotsc,i_m].
\end{align*}

\end{proof}

\begin{lemma}\label{lem:5.2}
Assume $a,r,t \geq 0$ and $s \geq 1$ are such that $2a + s + r + t = m$. Then, in \eqref{eq:bracket-def} we have for $r = 0$,
    \begin{equation}
       \operatorname{ad}(X)([\eta^ax^s\partial^t ; i_1,\dotsc,i_m]) = \frac{i}{\sqrt{2}}[\eta^ax^{s-1}\gamma\partial^t; i_1,\dotsc,i_m]
    \end{equation}
    and for $r = 1$,
    \begin{equation}
    \begin{aligned}
        \operatorname{ad}(X)([\eta^ax^s\gamma\partial^t; i_1,\dotsc,i_m]) 
        &= \frac{2i(a+1)}{\sqrt{2}}[\eta^{a+1}x^{s-1}\partial^{t};i_1,\dotsc,i_m] \\&+ \frac{2i(t+1)}{\sqrt{2}}[\eta^ax^s\partial^{t+1};i_1,\dotsc,i_m]
    \end{aligned}
    \end{equation}
\end{lemma}

\begin{proof}
    See Appendix \ref{appendix:brackets}.
\end{proof}

\begin{lemma}\label{lem:5.3}
Let $k \geq 0$ and $r \in \{0,1\}$. Then
    \begin{equation}
        \operatorname{ad}(X)^{2k+r}([x^m ; i_1,\dotsc,i_m]) = (-1)^{k}k!\bigg(\frac{i}{\sqrt{2}}\bigg)^r\sum_{t=0}^{k}[\eta^tx^{m-k-t-r}\gamma^r\partial^{k-t}; i_1,\dotsc,i_m].
    \end{equation}
\end{lemma}
\begin{proof}
    We first prove the following by induction on $k$. Define $D := \frac{\sqrt{2}}{i}\operatorname{ad}(X)$. Fix $i_1,\dotsc,i_m \in [n]$. Now, for the remainder of the proof we will omit the $i_1,\dotsc,i_m$, writing $[\eta^{a}x^{s}\gamma^{r}\partial^{t}] := [\eta^{a}x^{s}\gamma^{r}\partial^{t};i_1,\dotsc,i_m]$.
    Then
    \begin{align*}
        D^{2k}([x^m]) = 2^kk!\sum_{t=0}^{k}[\eta^{t}x^{m-k-t}\partial^{k-t}].
    \end{align*}
    For $k = 0$, we have
    \begin{align*}
        (-1)^{0}0![\eta^{0}x^{m-0-0}\partial^{0-0}] = [x^{m}] = D^{0}([x^m]).
    \end{align*}
    Now suppose we have the above formula for some $k \geq 0$.
    Then
    \begin{align*}
        &D^{2}D^{2k}([x^{m}]) \\
        &= 2^kk!\sum_{t=0}^{k}D^2[\eta^{t}x^{m-k-t}\partial^{k-t}]\\
        &= 2^kk!\sum_{t=0}^{k}D[\eta^{t}x^{m-k-t-1}\gamma\partial^{k-t}]\\
        &= 2^kk!\sum_{t=0}^{k}2(t+1)[\eta^{t+1}x^{m-k-t-2}\partial^{k-t}] + 2^kk!\sum_{t=0}^{k}2(k-t+1)[\eta^{t}x^{m-k-t-1}\partial^{k-t+1}]\\
        &= 2^{k+1}k!\sum_{t=1}^{k+1}(t)[\eta^{t}x^{m-k-t-1}\partial^{k-t+1}] \\
        &+ 2^{k+1}k!\sum_{t=1}^{k}(k-t+1)[\eta^{t}x^{m-k-t-1}\partial^{k-t+1}] + 2^{k+1}k!(k+1)[x^{m-k-1}\partial^{k+1}]\\
        &= 2^{k+1}k!\sum_{t=1}^{k}(t)[\eta^{t}x^{m-k-t-1}\partial^{k-t+1}] + 2^{k+1}(k+1)![\eta^{k+1}x^{m-2k-2}]\\
        &+ 2^{k+1}k!\sum_{t=1}^{k}(k-t+1)[\eta^{t}x^{m-k-t-1}\partial^{k-t+1}] + 2^{k+1}k!(k+1)[x^{m-k-1}\partial^{k+1}]\\
        &= 2^{k+1}(k+1)!\sum_{t=0}^{k+1}[\eta^{t}x^{m-k-t-1}\partial^{k+1-t}].
    \end{align*}

Furthermore, for all $k \geq 0$,
\begin{align*}
    D^{2k+1}([x^m]) = 2^kk!\sum_{t=0}^{k}D[\eta^{t}x^{m-k-t}\partial^{k-t}] = 2^{k}k!\sum_{t=0}^{k}[\eta^{t}x^{m-k-t-1}\gamma\partial^{k-t}].
\end{align*}

As such, for $r \in \{0,1\}$ and $k \geq 0$,
\begin{align*}
    \operatorname{ad}(X)^{2k+r}([x^m]) = \bigg(\frac{i}{\sqrt{2}}\bigg)^{2k+r}D^{2k+r}([x^m]) = (-1)^kk!\bigg(\frac{i}{\sqrt{2}}\bigg)^{r}\sum_{t=0}^{k}[\eta^{t}x^{m-k-t-r}\gamma^{r}\partial^{k-t}].
\end{align*}

\end{proof}

\begin{thm}\label{thm:product}
Write $\Tilde{F} := x_kx^k$ and $\Tilde{Y} := x_k\gamma^k$.
Then
\begin{equation}
\begin{aligned}
    &\widehat{x^{i_1}}\dotsb\widehat{x^{i_m}} = (H+1)\dotsb (H+m)(x^{i_1}\dotsb x^{i_m})\\
    &+\sum_{q=1}^{m}\frac{\prod_{k=q+1}^{m}(H+k)}{2^q}\Tilde{F}^{q-1}\bigg(\Tilde{Y}\sum_{t=0}^{q-1}[\eta^{t}x^{m-q-t}\gamma\partial^{q-1-t}] + \Tilde{F}\sum_{t=0}^{q}[\eta^{t}x^{m-q-t}\partial^{q-t}]\bigg)+I'
\end{aligned}
\end{equation}
\end{thm}

\begin{proof}
    First, recall that
    \begin{align*}
        \widehat{x^{i_1}}\dotsb\widehat{x^{i_m}} = (H+1)\dotsb(H+m)\sum_{k=0}^{\infty}\varphi_{k}(H)Y^{k}\operatorname{ad}(X)^{k}(x^{i_1}\dotsb x^{i_m}) + I'
    \end{align*}

    Note that by Lemma \ref{lem:5.3}, 
    \begin{align*}
        \operatorname{ad}(X)^{2m+1}([x^m;i_1,\dotsc,i_m]) = (-1)^{m}m!\frac{i}{\sqrt{2}}\sum_{t=0}^{m}[\eta^{t}x^{-t-1}\gamma\partial^{m-t}].
    \end{align*}
    However, since the power of $x$ is negative, we have that $\operatorname{ad}(X)^{k}([x^{m};i_1,\dotsc,i_m]) = 0$ for all $k > 2m$.
    So,
    \begin{align*}
        \widehat{x^{i_1}}\dotsb\widehat{x^{i_m}} &= (H+1)\dotsb(H+m)\sum_{k=0}^{2m}\varphi_{k}(H)Y^{k}\operatorname{ad}(X)^{k}(x^{i_1}\dotsb x^{i_m}) + I'\\
        &= (H+1)\dotsb(H+m)(x^{i_1}\dotsb x^{i_m})\\
        &+ (H+1)\dotsb (H+m)\sum_{q=1}^{m}\varphi_{2q-1}(H)Y^{2q-1}\operatorname{ad}(X)^{2q-1}([x^{m};i_1,\dotsc,i_m])\\
        &+ (H+1)\dotsb (H+m)\sum_{q=1}^{m}\varphi_{2q}(H)Y^{2q}\operatorname{ad}(X)^{2q}([x^{m};i_1,\dotsc,i_m]) + I'
    \end{align*}

    Then, by \eqref{eq:phi-odd} and \eqref{eq:phi-even},
    \begin{align*}
        &\varphi_{2q-1}(H) = \frac{-1}{(q-1)!}\frac{1}{(H+1)\dotsb (H+q)}\\
        &\varphi_{2q}(H) = \frac{1}{q!}\frac{1}{(H+1)\dotsb (H+q)}.
    \end{align*}
\end{proof}

Additionally, writing $Y = \frac{i}{\sqrt{2}}x_k\gamma^k$ and omitting $i_1,\dotsc,i_m$, we have

\begin{align*}
&(H+1)\dotsb (H+m)\sum_{q=0}^{m}\varphi_{2q-1}(H)Y^{2q-1}\operatorname{ad}(X)^{2q-1}([x^{m}])\\
&=\sum_{q=1}^{m}\frac{-1}{(q-1)!}\frac{\prod_{t=0}^{m-1}(H+1+t)}{\prod_{t=0}^{q-1}(H+1+t)}\bigg(\frac{i}{\sqrt{2}}\bigg)^{2q-1}(x_k\gamma^k)^{2q-1}(-1)^{q-1}(q-1)!\bigg(\frac{i}{\sqrt{2}}\bigg)\sum_{t=0}^{q-1}[\eta^{t}x^{m-q-t}\gamma\partial^{q-1-t}]\\
&= \sum_{q=1}^{m}\bigg(\frac{-1}{2}\bigg)^{q}(-1)^{q}(H+q+1)\dotsb(H+m)(x_k\gamma^k)^{2q-1}\sum_{t=0}^{q-1}[\eta^{t}x^{m-q-t}\gamma\partial^{q-1-t}]\\
&= \sum_{q=1}^{m}\frac{1}{2^q}(H+q+1)\dotsb(H+m)(x_k\gamma^k)^{2q-1}\sum_{t=0}^{q-1}[\eta^{t}x^{m-q-t}\gamma\partial^{q-1-t}].
\end{align*}

Likewise,

\begin{align*}
    &(H+1)\dotsb(H+m)\sum_{q=0}^{m}\varphi_{2q}(H)Y^{2q}\operatorname{ad}(X)^{2q}([x^m])\\
    &= \sum_{q=1}^{m}\frac{1}{q!}\frac{\prod_{t=0}^{m-1}(H+1+t)}{\prod_{t=0}^{q-1}(H+1+t)}\bigg(\frac{i}{\sqrt{2}}\bigg)^{2q}(x_k\gamma^k)^{2q}(-1)^{q}(q)!\sum_{t=0}^{q}[\eta^{t}x^{m-q-t}\partial^{q-t}]\\
    &= \sum_{q=1}^{m}\frac{1}{2^q}(H+q+1)\dotsb(H+m)(x_k\gamma^k)^{2q}\sum_{t=0}^{q}[\eta^{t}x^{m-q-t}\partial^{q-t}].
\end{align*}

Putting it all together,
\begin{equation}
\begin{aligned}
    \widehat{x^{i_1}}\dotsb\widehat{x^{i_m}} &= (H+1)\dotsb(H+m)(x^{i_1}\dotsb x^{i_m})\\
    &+  \sum_{q=1}^{m}\frac{1}{2^q}(H+q+1)\dotsb(H+m)(x_k\gamma^k)^{2q-1}\sum_{t=0}^{q-1}[\eta^{t}x^{m-q-t}\gamma\partial^{q-1-t}]\\
    &+\sum_{q=1}^{m}\frac{1}{2^q}(H+q+1)\dotsb(H+m)(x_k\gamma^k)^{2q}\sum_{t=0}^{q}[\eta^{t}x^{m-q-t}\partial^{q-t}]+I'.
\end{aligned}
\end{equation}

We can further simplify using the fact that $(x_k\gamma^k)^2 = x_kx^k$. Write $\Tilde{F} := x_kx^k$ and $\Tilde{Y} := x_k\gamma^k$.

Then
\begin{align*}
    &\widehat{x^{i_1}}\dotsb\widehat{x^{i_m}} = (H+1)\dotsb (H+m)(x^{i_1}\dotsb x^{i_m})\\
    &+\sum_{q=1}^{m}\frac{1}{2^q}(H+q+1)\dotsb(H+m)\Tilde{F}^{q-1}\bigg(\Tilde{Y}\sum_{t=0}^{q-1}[\eta^{t}x^{m-q-t}\gamma\partial^{q-1-t}] + \Tilde{F}\sum_{t=0}^{q}[\eta^{t}x^{m-q-t}\partial^{q-t}]\bigg)+I'
\end{align*}

\section{Generating Solutions}\label{sec:solutions}

For brevity, write
\begin{align}
    S_q^{0} := \sum_{t=0}^{q}[\eta^{t}x^{m-q-t}\partial^{q-t}] \qquad
    S_q^{1} := \sum_{t=0}^{q-1}[\eta^{t}x^{m-q-t}\gamma\partial^{q-1-t}].
\end{align}

Then we have
\begin{equation}
\begin{aligned}
    &\widehat{x^{i_1}}\dotsb\widehat{x^{i_m}} = (H+1)\dotsb (H+m)(x^{i_1}\dotsb x^{i_m})\\
    &+\sum_{q=1}^{m}\frac{\prod_{k=q+1}^{m}(H+k)}{2^q}\Tilde{F}^{q-1}\bigg(\Tilde{Y}S_q^{1} + \Tilde{F}S_q^{0}\bigg) + I'
\end{aligned}
\end{equation}

\begin{thm}\label{thm:solution-acting-on-1}
    \begin{equation}
        \begin{aligned}
    &\widehat{x^{i_1}}\dotsb\widehat{x^{i_m}}\cdot 1 = (-1)^{m}\bigg(\prod_{k=1}^{m}(\frac{n}{2}+m-k)\bigg)(x^{i_1}\dotsb x^{i_m})\\
    &+ \sum_{q=1}^{\lfloor\frac{m+1}{2}\rfloor}\frac{\prod_{k=q+1}^{m}(-\frac{n}{2} + k - m)}{2^q}\tilde{F}^{q-1}\tilde{Y}[\eta^{q-1}x^{m-2q+1}\gamma ; i_1,\dotsc,i_m]\\
    &+ \sum_{q=1}^{\lfloor\frac{m}{2}\rfloor}\frac{\prod_{k=q+1}^{m}(-\frac{n}{2} + k - m)}{2^q}\Tilde{F}^{q}[\eta^{q}x^{m-2q};i_1,\dotsc,i_m].
        \end{aligned}
    \end{equation}
\end{thm}
\begin{proof}
    Note that $Hx^k = x^k(H-1)$ and $H\partial^k = \partial^k(H+1)$, and that $H$ commutes with the $\gamma^k$.
    It follows that $HS_q^{0} = S_q^{0}(H-m+2q)$ and $HS_q^{1} = S_q^{1}(H-m+2q-1)$. Since $\Tilde{F}$ and $\Tilde{Y}$ have degree $2$ and $1$ respectively, we obtain $H\Tilde{F}^{q-1}\Tilde{Y}S_q^{1} = \Tilde{F}^{q-1}\Tilde{Y}S_q^{1}(H-m)$ and $H\Tilde{F}^{q}S_q^{0} = \Tilde{F}^{q}S_q^{0}(H-m)$.
    Hence
    \begin{align*}
        &\widehat{x^{i_1}}\dotsb\widehat{x^{i_m}} = (x^{i_1}\dotsb x^{i_m})(H-m+1)(H-m+2)\dotsb (H)\\
        &+\sum_{q=1}^{m}\frac{1}{2^q}\Tilde{F}^{q-1}\bigg(\Tilde{Y}S_q^{1} + \Tilde{F}S_q^{0}\bigg)\prod_{k=q+1}^{m}(H+k-m) + I'.
    \end{align*}

Note that for a polynomial $\varphi(t) \in \mathbb{C}[t]$, $\varphi(H)\cdot 1 = \varphi(\frac{-n}{2})$. 

Hence
\begin{align*}
    &\widehat{x^{i_1}}\dotsb\widehat{x^{i_m}}\cdot 1 = (-1)^{m}\bigg(\prod_{k=1}^{m}(\frac{n}{2}+m-k)\bigg)(x^{i_1}\dotsb x^{i_m})\\
    &+ \sum_{q=1}^{m}\frac{1}{2^q}\tilde{F}^{q-1}\bigg(\prod_{k=q+1}^{m}(-\frac{n}{2} + k - m)\bigg)\tilde{Y}S_q^{1}\cdot 1 + \frac{1}{2^q}\Tilde{F}^{q-1}\bigg(\prod_{k=q+1}^{m}(-\frac{n}{2} + k - m)\bigg)\tilde{F}S_q^{0}\cdot 1.
\end{align*}

We now consider $S_q^{1}\cdot 1$ and $S_q^{0}\cdot 1$.

First note that in the following summation, a summand is zero if $m < q+t$. As such, nonzero terms occur for $0 \leq t \leq \operatorname{min}(q,m-q)$
\begin{align*}
    S_q^{0} = \sum_{t\leq \operatorname{min}(q,m-q)}[\eta^{t}x^{m-q-t}\partial^{q-t}; i_1,\dotsc,i_m].
\end{align*}

Likewise,
\begin{align*}
    S_q^{1} = \sum_{t \leq \operatorname{min}(q,m-q)}[\eta^{t}x^{m-q-t}\gamma\partial^{q-1-t}; i_1,\dotsc,i_m].
\end{align*}

As such,
\begin{align}
    &S_q^{0}\cdot 1 = [\eta^{q}x^{m-2q};i_1,\dotsc,i_m]\\
    &S_q^{1}\cdot 1 = [\eta^{q-1}x^{m-2q+1}\gamma;i_1,\dotsc,i_m].
\end{align}
Note that the above terms are nonzero only if $2q \leq m$ or $2q-1 \leq m$. 

Finally, since terms with $\partial$'s on the right will annihilate one, the only terms that are nonzero are those consisting entirely of $\eta$'s, $x$'s, or $\gamma$'s. As such, we have the following:

\begin{align*}
    &\widehat{x^{i_1}}\dotsb\widehat{x^{i_m}}\cdot 1 = (-1)^{m}\bigg(\prod_{k=1}^{m}(\frac{n}{2}+m-k)\bigg)(x^{i_1}\dotsb x^{i_m})\\
    &+ \sum_{q=1}^{\lfloor\frac{m+1}{2}\rfloor}\frac{\prod_{k=q+1}^{m}(-\frac{n}{2} + k - m)}{2^q}\tilde{F}^{q-1}\tilde{Y}[\eta^{q-1}x^{m-2q+1}\gamma ; i_1,\dotsc,i_m]\\
    &+ \sum_{q=1}^{\lfloor\frac{m}{2}\rfloor}\frac{\prod_{k=q+1}^{m}(-\frac{n}{2} + k - m)}{2^q}\Tilde{F}^{q}[\eta^{q}x^{m-2q};i_1,\dotsc,i_m].
\end{align*}

\end{proof}

\begin{cor}\label{cor:solution-on-clifford-monomial}
    Let $1 \leq \alpha_1 \leq \dotsb \leq \alpha_r \leq n$. Write $p := \gamma^{\alpha_1}\dotsb \gamma^{\alpha_r}$. Then
    \begin{align*}
        &\widehat{x^{i_1}}\dotsb\widehat{x^{i_m}}\cdot p =  (\widehat{x^{i_1}}\dotsb\widehat{x^{i_m}}\cdot 1)p
    \end{align*}
\end{cor}

\begin{proof}
    In the proof of Theorem \ref{thm:solution-acting-on-1}, we used the fact that for a polynomial $\varphi(t) \in \mathbb{C}[t]$, $\varphi(H)\cdot 1 = \varphi(\frac{-n}{2})$. It follows by a similar argument that $\varphi(H)\cdot p = \varphi(\frac{-n}{2})p$. Hence
    \begin{align*}
        &\widehat{x^{i_1}}\dotsb\widehat{x^{i_m}}\cdot p = (-1)^{m}\bigg(\prod_{k=1}^{m}(\frac{n}{2}+m-k)\bigg)(x^{i_1}\dotsb x^{i_m})p\\
    &+ \sum_{q=1}^{m}\frac{1}{2^q}\tilde{F}^{q-1}\bigg(\prod_{k=q+1}^{m}(-\frac{n}{2} + k - m)\bigg)\tilde{Y}S_q^{1}\cdot p + \frac{1}{2^q}\Tilde{F}^{q-1}\bigg(\prod_{k=q+1}^{m}(-\frac{n}{2} + k - m)\bigg)\tilde{F}S_q^{0}\cdot p.
    \end{align*}

    However, since $\partial^i \cdot p = \partial^i \cdot 1 = 0$, we have that
    \begin{align}
    &S_q^{0}\cdot p = [\eta^{q}x^{m-2q};i_1,\dotsc,i_m]p\\
    &S_q^{1}\cdot p = [\eta^{q-1}x^{m-2q+1}\gamma;i_1,\dotsc,i_m]p.
\end{align}
Hence
\begin{align*}
        &\widehat{x^{i_1}}\dotsb\widehat{x^{i_m}}\cdot p = (-1)^{m}\bigg(\prod_{k=1}^{m}(\frac{n}{2}+m-k)\bigg)(x^{i_1}\dotsb x^{i_m})\gamma^{\alpha_1}\dotsb \gamma^{\alpha_r}\\
    &+ \sum_{q=1}^{\lfloor\frac{m+1}{2}\rfloor}\frac{\prod_{k=q+1}^{m}(-\frac{n}{2} + k - m)}{2^q}\tilde{F}^{q-1}\tilde{Y}[\eta^{q-1}x^{m-2q+1}\gamma ; i_1,\dotsc,i_m]p\\
    &+ \sum_{q=1}^{\lfloor\frac{m}{2}\rfloor}\frac{\prod_{k=q+1}^{m}(-\frac{n}{2} + k - m)}{2^q}\Tilde{F}^{q}[\eta^{q}x^{m-2q};i_1,\dotsc,i_m]p.
    \end{align*}
\end{proof}

\section{Irreducibility}\label{sec:irreducibility}

So far, we have only considered the action of $\mathcal{A}$ on $V$ with the usual Weyl-algebra action in addition to multiplication by elements of the Clifford Algebra. However, given a solution to the Dirac equation, it follows that multiplying on the right by elements of the Clifford Algebra fixes $V^{+}$, the set of solutions. In this section, we will consider the space of Clifford-value polynomials $V$ as a representation of $Z_n\otimes C(n)^{\operatorname{op}}$, where $C(n)^{\operatorname{op}}$ acts on the right by right-multiplication and $Z_n$ on the left by the action described in previous sections. Doing this, we show that for $n \neq 2$, the space of Clifford-valued polynomial solutions to the Dirac equation $V^{+}$ are  (irreducible) representations corresponding to those of the Clifford Algebra. In particular, whenever $n > 2$ is even, $V^{+}$ is irreducible. 

We wish to deploy the extremal projector over $V$. To do so, we need $X$ to eventually kill every element in $V$ so as to make the action of the projector well-defined.
\begin{lemma}
   The element $X$ acts locally finitely on $V$.
\end{lemma}
\begin{proof}
    Note that $X$ acts as a differential operator, reducing the degree of a Clifford-value polynomial by one. As such, a polynomial of degree $d$ is annihilated by $X^{d+1}$.
\end{proof}

Now, recall that $H$ acts by the degree operator on polynomials. Since elements of the form $H + m$ may act by zero, action by the rational functions of $H$ of the localized Clifford-Weyl algebra is not well-defined. To correct for this, we define a truncated extremal projector without $H$ operators in the denominator.

\begin{lemma}
    For every $\varphi \in V$, there exists some $N \geq 0$ where $X^{N}\cdot \varphi = 0$. Write $k = \lceil N/2 \rceil$. Then we consider the $N$-truncated extremal projector defined by
    \begin{align*}
        P_N := \sum_{n=0}^{N}\psi_n({H})Y^{n}X^{n} \in \operatorname{End}(V),
    \end{align*}
    where
    \begin{align*}
        \psi_n(t) = (t+1)\dotsb (t+k)\varphi_n(t) \in \mathbb{C}[t].
    \end{align*}
    Then
    \begin{itemize}
        \item[1.] $XP_N = P_NY = 0$.
        \item[2.] $P_N(v) = (H+1)\dotsb (H+k)v$ for all $v \in V^{+}$.
    \end{itemize}
\end{lemma}
\begin{proof}
    With $v$ in the $\mathcal{A}$-representation $V$, we have the following from an earlier section:
    \begin{align*}
        XP_N(v) &= \sum_{n = 0}^{N}X\psi_n(H)Y^nX^n(v)\\
        &= \sum_{n=0}^{N}\psi_n(H-1)XY^{n}X^{n}(v)\\
        &= \sum_{n=0}^{N}\psi_n(H-1)(XY^{n})X^{n}(v)\\
        &= \psi_0(H-1)X(v)\\
        &\quad +\sum_{n=1}^{N}\psi_n(H-1)\bigg((-1)^{n}Y^{n}X + \kappa_n(H-1)Y^{n-1}\bigg)X^n(v).
    \end{align*}
    Note since $X^{N+1}(v) = 0$, 
    \begin{align*}
        XP_N(v) &= \psi_0(H-1)X(v) + \sum_{n=1}^{N}(-1)^{n}\psi_n(H-1)Y^{n}X^{n+1}(v) + \sum_{n=1}^{N}\psi_n(H-1)\kappa_n(H-1)Y^{n-1}X^{n}(v)\\
        &= \sum_{n=0}^{N-1}(-1)^{n}\psi_n(H-1)Y^{n}X^{n+1}(v) + \sum_{n=0}^{N-1}\psi_{n+1}(H-1)\kappa_{n+1}(H-1)Y^{n}X^{n+1}(v)\\
        &= \sum_{n=0}^{N-1}\bigg((-1)^{n}\psi_n(H-1) + \kappa_{n+1}(H-1)\psi_{n+1}(H-1)\bigg)Y^{n}X^{n+1}(v).
    \end{align*}
    But by definition of the polynomials $(\varphi_n(H))_{n \in \mathbb{N}}$, we have
    \begin{align*}
        &(-1)^{n}\psi_n(H-1) + \kappa_{n+1}(H-1)\psi_{n+1}(H-1) \\
        &= (H+1)\dotsb (H+k)\bigg( (-1)^{n}\varphi_n(H-1) + \kappa_{n+1}(H-1)\varphi_{n+1}(H-1)\bigg) = 0.
    \end{align*}
    Hence $XP_N(v) = 0$.\\

    By a similar argument, $P_NY(v) = 0$.

    Note if $v \in V^{+}$, then $X(v) = 0$ by definition. As such, for any $N \geq 1$,
    \begin{align*}
        P_N(v) = \sum_{n=0}^{N}\psi_n(H)Y^{n}X^{n}(v)
        = \psi_0(H)v = (H+1)\dotsb (H+\lceil N/2 \rceil)v.
    \end{align*}
        
\end{proof}

We now discuss the irreducible representations of the $V^{+}$. First, we show that $1 \in V^{+}$ is a cyclic element. To this end, we make the following observation.

\begin{definition}\label{def:gamma-tilde}
    For $i,j_1,\dotsc,j_r \in [n]$, note that
\begin{align*}
     \widehat{\gamma^i}\cdot \gamma^{j_1}\dotsb \gamma^{j_r} &= ((H+1)\gamma^i + \gamma^jx_j\partial^i)\cdot (\gamma^{j_1}\dotsb \gamma^{j_r})\\
     &= (\frac{-n}{2}+1)\gamma^i\gamma^{j_1}\dotsb \gamma^{j_r}.
\end{align*}
Since $n \neq 2$, we define $\widetilde{\gamma}^i := \frac{1}{1 - \frac{n}{2}}\widehat{\gamma^i}$.
Additionally, for any element $p \in \operatorname{Cl}(n)$, write:
\begin{align*}
    p = \sum_{\alpha_1 \le \dotsb \le \alpha_r \in \{0,1\}}c_{\alpha}\gamma^{\alpha_1}\dotsb \gamma^{\alpha_r}.
\end{align*}
We define the following operator
\begin{align*}
    \widetilde{p} := \sum_{\alpha_1 \le \dotsb \le \alpha_r}c_{\alpha}\widetilde{\gamma}^{\alpha_1}\dotsb \widetilde{\gamma}^{\alpha_r}.
\end{align*}
The above operator is defined such that $\widetilde{p}\cdot 1 = p$.    
\end{definition}

Now, we show that any solution in $V^{+}$ can be obtained by acting by elements of $Z_n$ on $1 \in V^{+}$ (compare \cite[Section~7.5.5]{zhelobenkoRepresentationsReductiveLie1994}).

\begin{prop}\label{prp:CyclicV+}
    Let $\phi \in V^{+}$ be a Clifford-valued polynomial of degree $m$. Let $\phi_0,\phi_1,\dotsc,\phi_m$ denote the homogeneous components of $\phi$.
    We define an operator $\hat{\varphi} \in Z$ as follows. Write 
    \begin{align*}
         \phi_d = \sum_{1\le \alpha_1\le\cdots\le\alpha_d\le n}x^\alpha p_\alpha,
    \end{align*}
    where $x^\alpha =x^{\alpha_1}\cdots x^{\alpha_d}$ and $p_{\alpha} \in \operatorname{Cl}(n)$.
    We define
    \begin{align*}
        \widehat{\phi_d} = \sum_{\alpha \in \mathbb{N}^{n}}\frac{1}{(-1)^{d}\bigg(\prod_{k=1}^{d}(\frac{n}{2}+d-k)\bigg)}(\widehat{x^{\alpha_1}})\dotsb (\widehat{x^{\alpha_d}})\widetilde{p_{\alpha}},
    \end{align*}
    and then $\widehat{\phi} := \widehat{\phi_0} + \dotsb + \widehat{\phi_m}$.
    It follows that
    \begin{align*}
        \widehat{\phi}\cdot 1 = \phi.
    \end{align*}
\end{prop}

\begin{proof}
    Since the Dirac operator reduces the degree of a non-constant polynomial by exactly one, each homogeneous component of a polynomial solution is also a solution. Therefore, it suffices to show the above holds for each homogeneous component of $\phi$. To that end, assume $\phi$ is a homogeneous polynomial of degree $d$.

    We then have that
    \begin{align*}
        &\widehat{\phi}\cdot 1 = \sum_{1\le \alpha_1\le\cdots\le\alpha_d\le n}\frac{1}{(-1)^{d}\bigg(\prod_{k=1}^{d}(\frac{n}{2}+d-k)\bigg)}(\widehat{x^{\alpha_1}})\dotsb (\widehat{x^{\alpha_d}})\widetilde{p_{\alpha}}\cdot 1\\
        &= \sum_{1\le \alpha_1\le\cdots\le\alpha_d\le n}\frac{1}{(-1)^{d}\bigg(\prod_{k=1}^{d}(\frac{n}{2}+d-k)\bigg)}(\widehat{x^{\alpha_1}}\dotsb \widehat{x^{\alpha_d}})\cdot (p_{\alpha}).
    \end{align*}
    By Corollary \ref{cor:solution-on-clifford-monomial}, we have
    \begin{align*}
        \widehat{\phi}\cdot 1 = \sum_{1\le \alpha_1\le\cdots\le\alpha_d\le n}\frac{1}{(-1)^{d}\bigg(\prod_{k=1}^{d}(\frac{n}{2}+d-k)\bigg)}(\widehat{x^{\alpha_1}}\dotsb \widehat{x^{\alpha_d}}\cdot 1)(p_{\alpha}).
    \end{align*}
    In particular,
    \begin{align*}
        \widehat{\phi}\cdot 1 \in \sum_{1\le \alpha_1\le\cdots\le\alpha_d\le n}x^{\alpha_1}\dotsb x^{\alpha_d}p_{\alpha} + YV = \phi + YV.
    \end{align*}

    We thus have some $v \in V$ with
    \begin{align*}
        \widehat{\phi}\cdot 1 - \phi \in YV.
    \end{align*}
    Note by Theorem \ref{thm:solution-acting-on-1}, that $\widehat{\phi}\cdot 1 - \phi$ is a homogeneous element of degree $d$. As such, we have for any $\varphi(t) \in \mathbb{C}[t]$,
    \begin{align*}
        \varphi(H)\cdot (\widehat{\phi}\cdot 1 - \phi) = \varphi(-d - \frac{n}{2})(\widehat{\phi}\cdot 1 - \phi).
    \end{align*}
    Since $\widehat{\phi}\cdot 1 - \phi$ is a homogeneous element of degree $d$, we have that $X^{d+1}(\widehat{\phi}\cdot 1 - \phi) = 0$. Furthermore, since $\widehat{\phi}\cdot 1 - \phi \in V^{+}$, we have
    \begin{align*}
        P_{d}(\widehat{\phi}\cdot 1 - \phi) &= (H+1)\dotsb (H+\lceil d/2 \rceil)\cdot (\widehat{\phi}\cdot 1 - \phi)\\
        &= (-d - \frac{n}{2} + 1)\dotsb (-d - \frac{n}{2} + \lceil d/2 \rceil)(\widehat{\phi}\cdot 1 - \phi).
    \end{align*}
    Note that for any $1 \leq i \leq \lceil d/2 \rceil$, we have $(-d - \frac{n}{2} + i) = 0$ if and only if $i = d + \frac{n}{2}$. However, $i \leq \lceil d/2 \rceil \leq d < d + \frac{n}{2}$. As such,
    \begin{align*}
        P_d(\widehat{\phi}\cdot 1 - \phi) \in \mathbb{C}^{\times}(\widehat{\phi}\cdot 1 - \phi).
    \end{align*}

    At the same time, note that since $\widehat{\phi}\cdot 1 - \phi \in Yv$, we have $P_d(\widehat{\phi}\cdot 1 - \phi) = 0$.
    Hence $\widehat{\phi}\cdot 1 = \phi$.

\end{proof}

We now characterize the irreducibility of the representation $V^{+}$. Right multiplication by elements of $\operatorname{Cl}(n)$ commutes with the action of $A(2n|n)$ on $V$. Hence it preserves $V^{+}$ and commutes with the action of $Z_n$, so $V^{+}$ is a $Z_n\otimes \operatorname{Cl}(n)^{\operatorname{op}}$-module.

Recall the structure of the Clifford algebra $\operatorname{Cl}(n)$ \cite[Ch.~I]{lawsonSpinGeometry1989}. If $n$ is even, then $\operatorname{Cl}(n)$ is a simple algebra. If $n$ is odd, then the center of $\operatorname{Cl}(n)$ is spanned by $1$ and an element $\omega$ with $\omega^2=1$ (a suitable scalar multiple of the product of the elements of an orthonormal basis of $(\C^n,\eta)$). Putting $e_\pm=\frac{1}{2}(1\pm\omega)$, we then have $\operatorname{Cl}(n)=\operatorname{Cl}(n)e_+\oplus\operatorname{Cl}(n)e_-$, a direct sum of two simple algebras, and the nonzero two-sided ideals of $\operatorname{Cl}(n)$ are $\operatorname{Cl}(n)e_+$, $\operatorname{Cl}(n)e_-$ and $\operatorname{Cl}(n)$.

\begin{thm}\label{thm:irreducibility}
    Suppose $n \neq 2$, and regard $V^{+}$ as a $Z_n\otimes \operatorname{Cl}(n)^{\operatorname{op}}$-module.
    \begin{enumerate}[(i)]
        \item If $n$ is even, then $V^{+}$ is irreducible.
        \item If $n$ is odd, then $V^{+}=V^{+}e_+\oplus V^{+}e_-$, and both summands are irreducible submodules.
    \end{enumerate}
\end{thm}

\begin{proof}
    Let $W\subseteq V^{+}$ be a nonzero submodule, and let $\phi\in W$ be nonzero.
    Write $\phi$ in the form
    \begin{align*}
        \phi = \sum_{\alpha \in \mathbb{N}^{n}}(x^1)^{\alpha_1}\dotsb (x^n)^{\alpha_n}p_{\alpha},
    \end{align*}
    where $p_{\alpha}$ is an element of the Clifford algebra $\operatorname{Cl}(n)$.
    Choose $\beta \in \mathbb{N}^{n}$ with $p_{\beta} \neq 0$ such that $\beta_1 + \dotsb + \beta_n$ is maximal.
    It follows that
    \begin{align*}
        \widehat{\partial_1}^{\beta_1}\dotsb \widehat{\partial_n}^{\beta_n}\cdot (\phi) = (\beta_1)!\dotsb ( \beta_n)! p_{\beta}.
    \end{align*}
    That is, $p_{\beta} \in W$.
    By Definition \ref{def:gamma-tilde}, $\widetilde{\gamma}^i\cdot q=\gamma^i q$ for every $q\in\operatorname{Cl}(n)$, regarded as a constant solution. Hence $W\cap \operatorname{Cl}(n)$ is closed under left multiplication by $\operatorname{Cl}(n)$. It is also closed under right multiplication by $\operatorname{Cl}(n)$. Therefore $W$ contains the two-sided ideal $J$ of $\operatorname{Cl}(n)$ generated by $p_\beta$.
    On the other hand, for every $q\in\operatorname{Cl}(n)$ we have $Z_n\cdot q=(Z_n\cdot 1)q=V^{+}q$, by Proposition \ref{prp:CyclicV+}.

    (i) If $n$ is even, then $J=\operatorname{Cl}(n)$. Hence $1\in W$, and so $W\supseteq Z_n\cdot 1=V^{+}$.

    (ii) If $n$ is odd, then $J$ contains $e_+$ or $e_-$, so $W$ contains $V^{+}e_+$ or $V^{+}e_-$. Since $e_++e_-=1$ and $e_+e_-=e_-e_+=0$, we have $V^{+}=V^{+}e_+\oplus V^{+}e_-$, and both summands are submodules. They are nonzero, since $e_\pm\in V^{+}e_\pm$. Suppose now that $W\subseteq V^{+}e_+$. Every $w\in V^{+}e_+$ satisfies $w=we_+$, whereas $e_-e_+=0\neq e_-$. Hence $e_-\notin W$, so $W\supseteq V^{+}e_+$, that is, $W=V^{+}e_+$. The same argument applies to $V^{+}e_-$.
\end{proof}

\section{Appendix}\label{Appendix}

\subsection{Embedding of \texorpdfstring{$\mathfrak{osp}(1|2)$}{osp(1|2)} in the Weyl-Clifford Algebra}
\label{appendix:embedding}

\begin{proof}[Proof of Proposition \ref{prp:embedding}]
First, observe the following relations:
\begin{align}
    [\partial_i,x^j] &= \delta_{i}^j\\
    [\partial_i,x_j] &= [\partial_i,\eta_{kj}x^k] = \eta_{kj}\delta_{i}^{k} = \eta_{ij}\\
    [\partial^i,x^j] &= [\eta^{ki}\partial_k,x^j] = \eta^{ki}\delta_{k}^{j} = \eta^{ji}\\
    [\partial^i,x_j] &= [\eta^{ki}\partial_k,\eta_{mj}x^m] = \eta^{ki}\eta_{mj}\delta_{k}^{m} = \eta^{ki}\eta_{kj} = \delta_{j}^{i}.
\end{align}
Next, we calculate:
    \begin{align*}
        [H,Y] &= \frac{-i}{2\sqrt{2}}[\partial_ix^i + x^i\partial_i,\gamma^jx_j]\\
        &= \frac{-i}{2\sqrt{2}}\gamma^j[\partial_ix^i + x^i\partial_i,x_j]\\
        &= \frac{-i}{2\sqrt{2}}\gamma^j\bigg((\partial_ix^i + x^i\partial_i)(x_j) - (x_j)(\partial_ix^i + x^i\partial_i)\bigg)\\
        &= \frac{-i}{2\sqrt{2}}\gamma^j\bigg(\partial_ix^ix_j + x^i\partial_ix_j - x_j\partial_ix^i - x_jx^i\partial_i\bigg)\\
        &= \frac{-i}{2\sqrt{2}}\gamma^j\bigg(\partial_ix_jx^i + x^i(x_j\partial_i - [\partial_i,x_j]) - x_j\partial_ix^i - x_jx^i\partial_i\bigg)\\
        &=\frac{-i}{2\sqrt{2}}\gamma^j\bigg((x_j\partial_i - [\partial_i,x_j])x^i + x^i(x_j\partial_i - [\partial_i,x_j]) - x_j\partial_ix^i - x_jx^i\partial_i\bigg)\\
        &=\frac{-i}{2\sqrt{2}}\gamma^j\bigg(x_j\partial_ix^i - \eta_{ij}x^i + x^ix_j\partial_i - x^i\eta_{ij} - x_j\partial_ix^i - x_jx^i\partial_i\bigg)\\
        &=\frac{-i}{2\sqrt{2}}\gamma^j\bigg( -2x_j + x^ix_j\partial_i - x^ix_j\partial_i\bigg)\\
        &=\frac{-i}{2\sqrt{2}}\gamma^j\bigg(-2x_j\bigg) = \frac{-i}{\sqrt{2}}\gamma^jx_j = -Y.
    \end{align*}

    \begin{align*}
        [H,X] &= \frac{-i}{2\sqrt{2}}[\partial_ix^i + x^i\partial_i,\gamma^j\partial_j]\\
        &= \frac{-i\gamma^j}{2\sqrt{2}}(\partial_ix^i\partial_j + x^i\partial_i\partial_j - \partial_j\partial_ix^i - \partial_jx^i\partial_i)\\
        &= \frac{-i\gamma^j}{2\sqrt{2}}(\partial_i(\partial_jx^i - \delta_{j}^{i}) + ((\partial_jx^i - \delta_{j}^{i}))\partial_i - \partial_j\partial_ix^i - \partial_jx^i\partial_i)\\
        &=\frac{-i\gamma^j}{2\sqrt{2}}(\partial_i\partial_jx^i - \partial_i\delta_{j}^{i} + \partial_jx^i\partial_i - \partial_i\delta_{j}^{i} - \partial_j\partial_ix^i - \partial_jx^i\partial_i)\\
        &=\frac{-i\gamma^j}{2\sqrt{2}}(-2\partial_j) = \frac{i}{\sqrt{2}}\gamma^j\partial_j = X.
    \end{align*}

    \begin{align*}
        [Y,Y] &= \frac{-1}{2}[\gamma^ix_i,\gamma^jx_j]\\
        &= \frac{-1}{2}(\gamma^ix_i\gamma^jx_j + \gamma^jx_j\gamma^ix_i)\\
        &= \frac{-1}{2}(\gamma^i\gamma^j + \gamma^j\gamma^i)(x_ix_j)\\
        &= -2\frac{1}{2}\eta^{ij}x_ix_j = -2\frac{1}{2}x^jx_j = -2F.
    \end{align*}

    \begin{align*}
        [X,X] &= \frac{-1}{2}[\gamma^i\partial_i,\gamma^j\partial_j]\\
        &= \frac{-1}{2}(\gamma^i\partial_i\gamma^j\partial_j + \gamma^j\partial_j\gamma^i\partial_i)\\
        &= \frac{-1}{2}(\gamma^i\gamma^j + \gamma^j\gamma^i)\partial_i\partial_j\\
        &= \frac{-1}{2}(2\eta^{ij})\partial_i\partial_j = 2\frac{-1}{2}\partial_j\partial_j = 2E.
    \end{align*}

    \begin{align*}
        F = \frac{-1}{2}[Y,Y] = -Y^2.
    \end{align*}

    \begin{align*}
        E = \frac{1}{2}[X,X] = X^2.
    \end{align*}

    \begin{align*}
        [H,F] &= -[H,Y^2] \\
        &= -[H,Y]Y - Y[H,Y]\\
        &= 2Y^2\\
        &= -2F.
    \end{align*}

    \begin{align*}
        [H,E] &= [H,X^2]\\
        &= [H,X]X + X[H,X]\\
        &= 2X^2\\
        &=2E.
    \end{align*}

    \begin{align*}
        [Y,X] &= \frac{-1}{2}[\gamma^ix_i,\gamma^j\partial_j]\\
        &= \frac{-1}{2}(\gamma^ix_i\gamma^j\partial_j + \gamma^j\partial_j\gamma^ix_i)\\
        &= \frac{-1}{2}(\gamma^i\gamma^jx_i\partial_j + \gamma^j\gamma^i\partial_jx_i)\\
        &= \frac{-1}{2}(\gamma^i\gamma^j(\partial_jx^i - [x_i,\partial_j]) + \gamma^j\gamma^i\partial_jx_i)\\
        &= \frac{-1}{2}(\gamma^i\gamma^j\partial_jx_i + \gamma^i\gamma^j\eta_{ji} + \gamma^j\gamma^i\partial_jx_i)\\
        &= -\frac{1}{2}(\gamma^i\gamma^j + \gamma^j\gamma^i)\partial_jx_i - \frac{1}{2}\gamma^i\gamma^j\eta_{ij}\\
        &= -\frac{1}{2}(2\eta^{ij})\partial_jx_i - \frac{1}{4}(\gamma^i\gamma^j+\gamma^i\gamma^j)\eta_{ij}\\
        &= -\frac{1}{2}(2\eta^{ij})\partial_jx_i - \frac{1}{4}(2\eta^{ij})\eta_{ij}\\
        &= -\partial_ix^i - \frac{1}{2}\delta_{i}^{i} = -\frac{1}{2}(2\partial_ix^i - \delta_i^i) = \frac{-1}{2}(\partial_ix^i + x^i\partial_i) = H.
    \end{align*}

    \begin{align*}
        [Y,E] &= [Y,X^2]\\
        &=[Y,X]X + (-1)^{|Y||X|}X[Y,X]\\
        &= HX - XH\\
        &= [H,X] = X.
    \end{align*}

    \begin{align*}
        [X,F] &= -[X,Y^2]\\
        &= -[X,Y]Y - (-1)^{|X||Y|}Y[X,Y]\\
        &= -HY + YH\\
        &= -(HY - YH) = -[H,Y] = Y.
    \end{align*}

    \begin{align*}
        [E,F] &= [E,-Y^2]\\
        &= -[E,Y^2]\\
        &= -[E,Y]Y - Y[E,Y]\\
        &= -(-1)^{|E||Y|}[E,Y]Y - Y(-1)^{|E||Y|}[E,Y]\\
        &= [Y,E]Y + Y[Y,E]\\
        &= XY + YX = H.
    \end{align*}
\end{proof}

\subsection{Adjoint actions of \texorpdfstring{$\mathfrak{osp}(1|2)$}{osp(1|2)} on the localized Weyl-Clifford Algebra}
\label{appendix:adjoint-actions}

\begin{proof}[Proof of Proposition \ref{prp:2.1}]
Let $k \in \{1,\dotsc,n\}$ be arbitrary.

\begin{align*}
    [X,x^k] = [\frac{i}{\sqrt{2}}\gamma^i\partial_i,x^k] = \frac{i}{\sqrt{2}}\gamma^i[\partial_i,x^k] = \frac{i}{\sqrt{2}}\gamma^i\delta_i^{k} = \frac{i}{\sqrt{2}}\gamma^k.
\end{align*}

\begin{align*}
    [X,\gamma^k] = \frac{i}{\sqrt{2}}(\gamma^i\partial_i\gamma^k + \gamma^k\gamma^i\partial_i) = \frac{i}{\sqrt{2}}(\gamma^i\gamma^k + \gamma^k\gamma^i)\partial_i = \frac{i}{\sqrt{2}} 2\eta^{ik}\partial_i = 2\frac{i}{\sqrt{2}}\partial^k.
\end{align*}

\begin{align*}
    [X,\partial_k] = \frac{i}{\sqrt{2}}(\gamma^i\partial_i\partial_k - \partial_k\gamma^i\partial_i) = \frac{i}{\sqrt{2}}(\partial_k\gamma^i\partial_i - \partial_k\gamma^i\partial_i) = 0.
\end{align*}

\begin{align*}
    [Y,\partial_k] = \frac{i}{\sqrt{2}}(\gamma^ix_i\partial_k - \partial_k\gamma^ix_i) = \frac{i}{\sqrt{2}}(\gamma_ix^i\partial_k - \partial_k\gamma_ix^i) = \frac{i}{\sqrt{2}}\gamma_i(x^i\partial_k - \partial_kx^i) = \frac{i}{\sqrt{2}}\gamma_i(-\delta_k^i) = -\frac{i}{\sqrt{2}}\gamma_k.
\end{align*}

\begin{align*}
    [Y,\gamma^k] = \frac{i}{\sqrt{2}}(\gamma^ix_i\gamma^k + \gamma^k\gamma^ix_i) = \frac{i}{\sqrt{2}}x_i(\gamma^i\gamma^k+\gamma^k\gamma^i) = \frac{i}{\sqrt{2}}x_i(2\eta^{ik}) = 2\frac{i}{\sqrt{2}}x^k.
\end{align*}

\begin{align*}
    [Y,x^k] = \frac{i}{\sqrt{2}}(\gamma^ix_ix^k - x^k\gamma^ix_i) = \frac{i}{\sqrt{2}}(x^k\gamma^ix_i - x^k\gamma^ix_i) = 0.
\end{align*}

\begin{align*}
    [H,x^k] &= -\frac{1}{2}(\partial_ix^ix^k + x^i\partial_ix^k - x^k\partial_ix^i - x^kx^i\partial_i) \\
    &= -\frac{1}{2}(\partial_ix^kx^i - x^k\partial_ix^i + x^i\partial_ix^k - x^ix^k\partial_i) \\
    &= -\frac{1}{2}(\delta_i^kx^i + x^i\delta_i^k) = -x^k.
\end{align*}

\begin{align*}
    [H,\partial_k] &= -\frac{1}{2}(\partial_ix^i\partial_k + x^i\partial_i\partial_k - \partial_k\partial_ix^i - \partial_kx^i\partial_i)\\
    &= -\frac{1}{2}(\partial_ix^i\partial_k - \partial_i\partial_kx^i + x^i\partial_k\partial_i - \partial_kx^i\partial_i)\\
    &= -\frac{1}{2}(-\partial_i\delta_k^i - \delta_k^i\partial_i)\\
    &= \partial_k.
\end{align*}

\begin{align*}
    H\gamma^k = -\frac{1}{2}\partial_ix^i\gamma^k -\frac{1}{2}x^i\partial_i\gamma^k = -\frac{1}{2}\gamma^k\partial_ix^i -\frac{1}{2}\gamma^kx^i\partial_i = \gamma^kH.
\end{align*}

\end{proof}

\subsection{Brackets}
\label{appendix:brackets}

\begin{proof}[Proof of Lemma \ref{lem:5.2}]
    We use the following identity:
    \begin{align*}
        &[\eta^ax^s\gamma^r\partial^t ; i_1,\dotsc,i_m] \\
        &= \frac{1}{2^aa!s!t!}\sum_{\sigma \in S^m}(\eta^{a}x^s\gamma^r\partial^t)^{i_{\sigma\cdot(1,\dotsc,m)}}\\
        &= \frac{1}{2^aa!s!t!}\sum_{\sigma \in S^m}\bigg(\prod_{k=0}^{a-1}\eta^{i_{\sigma(2k+1)}i_{\sigma(2k+2)}}\bigg)\bigg(\prod_{k=1}^{s}x^{i_{\sigma(2a+k)}}\bigg)\bigg(\prod_{k=1}^{r}\gamma^{i_{\sigma(2a+s+1)}}\bigg)\bigg(\prod_{k=1}^{t}\partial^{i_{\sigma(2a+s+r+k)}}\bigg).
    \end{align*}
    First, consider when $r = 0$.
    Since $\operatorname{ad}(X)(\eta^{ij}) = \operatorname{ad}(X)(\partial^i) = 0$,
    \begin{align*}
        &\operatorname{ad}(X)([\eta^ax^s\partial^t ; i_1,\dotsc,i_m])\\
        &= \frac{1}{2^aa!s!t!}\sum_{\sigma \in S^m}\bigg(\prod_{k=0}^{a-1}\eta^{i_{\sigma(2k+1)}i_{\sigma(2k+2)}}\bigg)\operatorname{ad}(X)\bigg(\prod_{k=1}^{s}x^{i_{\sigma(2a+k)}}\bigg)\bigg(\prod_{k=1}^{t}\partial^{i_{\sigma(2a+s+k)}}\bigg)\\
        &= \frac{i}{\sqrt{2}}\cdot\frac{1}{2^aa!s!t!}\sum_{\sigma \in S^m}\sum_{c=1}^{s}\bigg(\prod_{k=0}^{a-1}\eta^{i_{\sigma(2k+1)}i_{\sigma(2k+2)}}\bigg)\bigg(\prod_{\substack{k \in [s] \\ k \neq c}}x^{i_{\sigma(2a+k)}}\bigg)\gamma^{i_{\sigma(2a+c)}}\bigg(\prod_{k=1}^{t}\partial^{i_{\sigma(2a+s+k)}}\bigg).
    \end{align*}
    Note that if $s = 0$, then $\operatorname{ad}(X)([\eta^a\partial^t; i_1,\dotsc,i_m]) = 0$. Hence the equation holds for $s=0$.
    Now for each $c \in [s]$, define a permutation $\sigma_c \in S^m$ by the cycle
    \begin{align*}
        \sigma_c := \begin{pmatrix}
            2a+r & 2a + r - 1 & \dotsb & 2a + c
        \end{pmatrix}
    \end{align*}
    Then
    \begin{align*}
        &\frac{i}{\sqrt{2}}\cdot\frac{1}{2^aa!s!t!}\sum_{\sigma \in S^m}\sum_{c=1}^{s}\bigg(\prod_{k=0}^{a-1}\eta^{i_{\sigma(2k+1)}i_{\sigma(2k+2)}}\bigg)\bigg(\prod_{\substack{k \in [s] \\ k \neq c}}x^{i_{\sigma(2a+k)}}\bigg)\gamma^{i_{\sigma(2a+c)}}\bigg(\prod_{k=1}^{t}\partial^{i_{\sigma(2a+s+k)}}\bigg)\\
        &=\frac{i}{\sqrt{2}}\cdot\frac{1}{2^aa!s!t!}\sum_{c=1}^{s}\sum_{\sigma \in S^m}\bigg(\prod_{k=0}^{a-1}\eta^{i_{\sigma(2k+1)}i_{\sigma(2k+2)}}\bigg)\bigg(\prod_{\substack{k \in [s] \\ k \neq c}}x^{i_{\sigma(2a+k)}}\bigg)\gamma^{i_{\sigma(2a+c)}}\bigg(\prod_{k=1}^{t}\partial^{i_{\sigma(2a+s+k)}}\bigg)\\
        &=\frac{i}{\sqrt{2}}\cdot\frac{1}{2^aa!s!t!}\sum_{c=1}^{s}\sum_{\sigma \in \sigma_c^{-1}S^m}\bigg(\prod_{k=0}^{a-1}\eta^{i_{\sigma(2k+1)}i_{\sigma(2k+2)}}\bigg)\bigg(\prod_{k=1}^{s-1}x^{i_{\sigma(2a+k)}}\bigg)\gamma^{i_{\sigma(2a+r)}}\bigg(\prod_{k=1}^{t}\partial^{i_{\sigma(2a+s+k)}}\bigg)\\
        &= \frac{i}{\sqrt{2}}\cdot\frac{s}{2^aa!s!t!}\sum_{\sigma \in S^m}(\eta^{a}x^{s-1}\gamma\partial^t)^{i_{\sigma\cdot(1,\dotsc,m)}} = \frac{i}{\sqrt{2}}\cdot[\eta^a x^{s-1} \gamma \partial^t ; i_1,\dotsc,i_m].
    \end{align*}
    On the other hand, suppose $r = 1$.
    Then by the super Leibniz rule, we have
    \begin{align*}
        &\operatorname{ad}(X)([\eta^ax^s\gamma\partial^t ; i_1,\dotsc,i_m])\\
        &= \frac{1}{2^aa!s!t!}\sum_{\sigma \in S^m}\bigg(\prod_{k=0}^{a-1}\eta^{i_{\sigma(2k+1)}i_{\sigma(2k+2)}}\bigg)\operatorname{ad}(X)\bigg(\prod_{k=1}^{s}x^{i_{\sigma(2a+k)}}\bigg)\gamma^{i_{\sigma(2a+s+1)}}\bigg(\prod_{k=1}^{t}\partial^{i_{\sigma(2a+s+1+k)}}\bigg)\\
        &+ \frac{1}{2^aa!s!t!}\sum_{\sigma \in S^m}\bigg(\prod_{k=0}^{a-1}\eta^{i_{\sigma(2k+1)}i_{\sigma(2k+2)}}\bigg)\bigg(\prod_{k=1}^{s}x^{i_{\sigma(2a+k)}}\bigg)\operatorname{ad}(X)(\gamma^{i_{\sigma(2a+s+1)}})\bigg(\prod_{k=1}^{t}\partial^{i_{\sigma(2a+s+1+k)}}\bigg).
    \end{align*}
    For the first term above, we use a similar argument as in the case for $r = 0$. We have
    \begin{align*}
        &\frac{1}{2^aa!s!t!}\sum_{\sigma \in S^m}\bigg(\prod_{k=0}^{a-1}\eta^{i_{\sigma(2k+1)}i_{\sigma(2k+2)}}\bigg)\operatorname{ad}(X)\bigg(\prod_{k=1}^{s}x^{i_{\sigma(2a+k)}}\bigg)\gamma^{i_{\sigma(2a+s+1)}}\bigg(\prod_{k=1}^{t}\partial^{i_{\sigma(2a+s+1+k)}}\bigg)\\
        &= \frac{i}{\sqrt{2}}\frac{1}{2^aa!s!t!}\sum_{\sigma \in S^m}\sum_{c = 1}^{s}\bigg(\prod_{k=0}^{a-1}\eta^{i_{\sigma(2k+1)}i_{\sigma(2k+2)}}\bigg)\bigg(\prod_{\substack{k \in [s]\\k\neq c}}x^{i_{\sigma(2a+k)}}\bigg)\gamma^{i_{\sigma(2a+c)}}\gamma^{i_{\sigma(2a+s+1)}}\bigg(\prod_{k=1}^{t}\partial^{i_{\sigma(2a+s+1+k)}}\bigg)\\
        &= \frac{i}{\sqrt{2}}\frac{s}{2^aa!s!t!}\sum_{\sigma \in S^m}\bigg(\prod_{k=0}^{a-1}\eta^{i_{\sigma(2k+1)}i_{\sigma(2k+2)}}\bigg)\bigg(\prod_{k=1}^{s-1}x^{i_{\sigma(2a+k)}}\bigg)\gamma^{i_{\sigma(2a+s)}}\gamma^{i_{\sigma(2a+s+1)}}\bigg(\prod_{k=1}^{t}\partial^{i_{\sigma(2a+s+1+k)}}\bigg)\\
        &= \frac{i}{\sqrt{2}}[\eta^{a}x^{s-1}\gamma^{2}\partial^{t}; i_1,\dotsc,i_m] = \frac{i}{\sqrt{2}}2(a+1)[\eta^{a+1}x^{s-1}\partial^{t} ; i_1,\dotsc,i_m].
    \end{align*}
    Note that when $s = 0$, this term is zero.
    For the second term, we have
    \begin{align*}
        &\frac{1}{2^aa!s!t!}\sum_{\sigma \in S^m}\bigg(\prod_{k=0}^{a-1}\eta^{i_{\sigma(2k+1)}i_{\sigma(2k+2)}}\bigg)\bigg(\prod_{k=1}^{s}x^{i_{\sigma(2a+k)}}\bigg)\operatorname{ad}(X)(\gamma^{i_{\sigma(2a+s+1)}})\bigg(\prod_{k=1}^{t}\partial^{i_{\sigma(2a+s+1+k)}}\bigg)\\
        &= \frac{2i}{\sqrt{2}}\frac{1}{2^aa!s!t!}\sum_{\sigma \in S^m}\bigg(\prod_{k=0}^{a-1}\eta^{i_{\sigma(2k+1)}i_{\sigma(2k+2)}}\bigg)\bigg(\prod_{k=1}^{s}x^{i_{\sigma(2a+k)}}\bigg)\partial^{i_{\sigma(2a+s+1)}}\bigg(\prod_{k=1}^{t}\partial^{i_{\sigma(2a+s+1+k)}}\bigg)\\
        &= \frac{2i}{\sqrt{2}}\frac{t+1}{2^aa!s!(t+1)!}\sum_{\sigma \in S^m}\bigg(\prod_{k=0}^{a-1}\eta^{i_{\sigma(2k+1)}i_{\sigma(2k+2)}}\bigg)\bigg(\prod_{k=1}^{s}x^{i_{\sigma(2a+k)}}\bigg)\bigg(\prod_{k=1}^{t+1}\partial^{i_{\sigma(2a+s+k)}}\bigg)\\
        &= \frac{2i(t+1)}{\sqrt{2}}[\eta^{a}x^{s}\partial^{t+1};i_1,\dotsc,i_m].
    \end{align*}
    So,
    \begin{align*}
        \operatorname{ad}(X)([\eta^ax^s\gamma\partial^t; i_1,\dotsc,i_m]) = \frac{2i(a+1)}{\sqrt{2}}[\eta^{a+1}x^{s-1}\partial^{t};i_1,\dotsc,i_m] + \frac{2i(t+1)}{\sqrt{2}}[\eta^ax^s\partial^{t+1};i_1,\dotsc,i_m].
    \end{align*}
    Note that if $s = 0$, then the above formula holds with $[\eta^{a+1}x^{s-1}\partial^t ;i_1,\dotsc,i_m] = 0$.
\end{proof}

\cleardoublepage
\newgeometry{left=2cm, right=2cm, top=4cm, bottom=4cm}
\section*{Author declarations}
\subsection*{Conflict of interest}
The authors have no conflicts to disclose.

\subsection*{Use of AI tools}
In preparing this version, the authors used Claude Opus 5.5 (Anthropic) to compare the paper with Zhelobenko's book \cite{zhelobenkoRepresentationsReductiveLie1994}, to add the corresponding citations, to revise the abstract and the introduction, to correct Theorem \ref{thm:solution-acting-on-1} and Corollary \ref{cor:solution-on-clifford-monomial}, to complete the proof of Theorem \ref{thm:irreducibility} for odd $n$, and to find and correct several other errors, including a sign error in the definition of $\widehat{\gamma^i}$. The authors have checked all changes and take full responsibility for the content.

\section*{Data availability}
Data sharing is not applicable to this article as no new data were created or analyzed in this study.

\printbibliography

@book{BrackxDS1982,
  title = {Clifford Analysis},
  author = {{Brackx, Fred} and Delanghe, {\relax Ri}chard and Sommen, Franciscus},
  year = {1982},
  volume = {76},
  pages = {308},
  publisher = {Pitman Books Limited},
  isbn = {0-273-08535-2},
  langid = {english}
}

@article{debieHarmonicTransvectorAlgebra2017,
  title = {The Harmonic Transvector Algebra in Two Vector Variables},
  author = {De Bie, Hendrik and Eelbode, David and Roels, Matthias},
  year = {2017},
  month = mar,
  journal = {Journal of Algebra},
  volume = {473},
  pages = {247--282},
  issn = {00218693},
  doi = {10.1016/j.jalgebra.2016.10.039},
  url = {https://linkinghub.elsevier.com/retrieve/pii/S0021869316304215},
  urldate = {2021-06-04},
  langid = {english}
}

@book{frappatDictionaryLieAlgebras2000a,
  title = {Dictionary on {{Lie}} Algebras and Superalgebras},
  author = {Frappat, Luc and Sciarrino, Antonino and Sorba, Paul},
  year = {2000},
  publisher = {Academic Press, Inc.},
  address = {San Diego, CA},
  isbn = {978-0-12-265340-7},
  langid = {english}
}

@article{hartwigDiagonalReductionAlgebra2022,
  ids = {hartwig2022diagonal,hartwigDiagonalReductionAlgebra2021,hartwigDiagonalReductionAlgebra2021b},
  title = {{Diagonal reduction algebra for $\mathfrak{osp}(1|2)$}},
  author = {Hartwig, Jonas T. and Williams II, Dwight Anderson},
  year = {2022},
  month = feb,
  journal = {Theoretical and Mathematical Physics},
  volume = {210},
  number = {2},
  pages = {155-171},
  publisher = {Russian Academy of Sciences, Steklov Mathematical Institute of Russian {\dots}},
  issn = {0040-5779, 1573-9333},
  doi = {10.1134/S0040577922020015},
  url = {https://link.springer.com/10.1134/S0040577922020015},
  langid = {english}
}

@misc{hartwigExactSolutionsKleinGordon2025,
  title = {Exact Solutions to the {{Klein--Gordon}} Equation via Reduction Algebras},
  author = {Hartwig, Jonas T. and Uhl, Lillian Ryan and Williams II, Dwight Anderson},
  year = {2025},
  month = jul,
  number = {arXiv:2507.04572},
  eprint = {2507.04572},
  primaryclass = {math},
  publisher = {arXiv},
  doi = {10.48550/arXiv.2507.04572},
  url = {http://arxiv.org/abs/2507.04572},
  urldate = {2025-07-09},
  archiveprefix = {arXiv},
  copyright = {All rights reserved}
}

@article{hartwigGhostCenterRepresentations2023,
  title = {{Ghost center and representations of the diagonal reduction algebra of $\mathfrak{osp}(1|2)$}},
  author = {Hartwig, Jonas T. and Williams II, Dwight Anderson},
  year = {2023},
  journal = {Journal of Geometry and Physics},
  volume = {187},
  pages = {104788},
  publisher = {Elsevier},
  issn = {0393-0440},
  doi = {10.1016/j.geomphys.2023.104788},
  url = {https://www.sciencedirect.com/science/article/pii/S0393044023000402},
  copyright = {All rights reserved}
}

@article{hartwigSymplecticDifferentialReduction2025,
  title = {Symplectic Differential Reduction Algebras and Generalized {{Weyl}} Algebras},
  author = {Hartwig, Jonas T. and Williams II, Dwight Anderson},
  year = {2025},
  month = jan,
  journal = {Symmetry, Integrability and Geometry: Methods and Applications},
  volume = {21},
  number = {001},
  issn = {18150659},
  doi = {10.3842/SIGMA.2025.001},
  url = {https://www.emis.de/journals/SIGMA/2025/001/},
  urldate = {2025-01-02},
  copyright = {All rights reserved}
}

@incollection{howeDualPairsPhysics1985,
  title = {Dual Pairs in Physics: {{Harmonic}} Oscillators, Photons, Electrons, and Singletons},
  booktitle = {Applications of {{Group Theory}} in {{Physics}} and {{Mathematical Physics}}},
  author = {Howe, Roger},
  editor = {Flato, Mosh and Sally, Paul and Zuckerman, Gregg},
  year = {1985},
  series = {Lectures in {{Applied Mathematics}}},
  volume = {21},
  pages = {179--207},
  publisher = {American Mathematical Society},
  address = {Providence, RI}
}

@article{khoroshkinDiagonalReductionAlgebras2010,
  title = {{Diagonal reduction algebras of $\mathfrak{gl}$ type}},
  author = {Khoroshkin, S. M. and Ogievetsky, O. V.},
  year = {2010},
  month = sep,
  journal = {Functional Analysis and Its Applications},
  volume = {44},
  number = {3},
  pages = {182-198},
  issn = {0016-2663, 1573-8485},
  doi = {10.1007/s10688-010-0023-0},
  url = {http://link.springer.com/10.1007/s10688-010-0023-0},
  urldate = {2023-12-16},
  langid = {english}
}

@article{khoroshkinMickelssonAlgebrasZhelobenko2008,
  title = {Mickelsson Algebras and {{Zhelobenko}} Operators},
  author = {Khoroshkin, S. M. and Ogievetsky, O. V.},
  year = {2008},
  month = mar,
  journal = {Journal of Algebra},
  volume = {319},
  number = {5},
  pages = {2113--2165},
  issn = {00218693},
  doi = {10.1016/j.jalgebra.2007.04.020},
  url = {https://linkinghub.elsevier.com/retrieve/pii/S0021869307002700},
  urldate = {2021-05-23},
  langid = {english}
}

@article{mickelssonStepAlgebrasSemisimple1973,
  title = {Step Algebras of Semi-Simple Subalgebras of {{Lie}} Algebras},
  author = {Mickelsson, Jouko},
  year = {1973},
  month = jun,
  journal = {Reports on Mathematical Physics},
  volume = {4},
  number = {4},
  pages = {307--318},
  issn = {00344877},
  doi = {10.1016/0034-4877(73)90006-2},
  url = {https://linkinghub.elsevier.com/retrieve/pii/0034487773900062},
  urldate = {2021-06-03},
  langid = {english}
}

@article{mudrovMickelssonAlgebrasHasse2025,
  title = {Mickelsson Algebras via {{Hasse}} Diagrams},
  author = {Mudrov, Andrey and Stukopin, Vladimir},
  year = {2025},
  month = apr,
  journal = {Algebras and Representation Theory},
  volume = {28},
  number = {2},
  pages = {353--367},
  issn = {1386-923X, 1572-9079},
  doi = {10.1007/s10468-024-10311-8},
  url = {https://link.springer.com/10.1007/s10468-024-10311-8},
  urldate = {2025-06-23},
  langid = {english}
}

@incollection{ryanCliffordAnalysis2011,
  title = {Clifford {{Analysis}}},
  booktitle = {Lectures on {{Clifford}} ({{Geometric}}) {{Algebras}} and {{Applications}}},
  author = {Ryan, John},
  editor = {Ablamowicz, RafaI and Sobczyk, Garret},
  year = {2011},
  pages = {53--89},
  publisher = {Birkh{\"a}user Boston},
  address = {Boston, MA}
}

@article{tolstoyExtremalProjectorsContragredient2011,
  ids = {tolstoyExtremalProjectorsContragredient2011a},
  title = {Extremal Projectors for Contragredient {{Lie}} (Super)Symmetries (Short Review)},
  author = {Tolstoy, V. N.},
  year = {2011},
  month = dec,
  journal = {Physics of Atomic Nuclei},
  volume = {74},
  number = {12},
  pages = {1747--1757},
  issn = {1063-7788, 1562-692X},
  doi = {10.1134/S1063778811070155},
  url = {http://link.springer.com/10.1134/S1063778811070155},
  urldate = {2021-05-23},
  langid = {english}
}

@article{vandenhomberghNoteMickelssonsStep1975,
  title = {A Note on {{Mickelsson}}'s Step {{Algebra}}},
  author = {{van den Hombergh}, Adolf Petrus},
  year = {1975},
  journal = {Indagationes Mathematicae (Proceedings)},
  volume = {78},
  number = {1},
  pages = {42--47},
  issn = {13857258},
  doi = {10.1016/1385-7258(75)90013-X},
  url = {https://linkinghub.elsevier.com/retrieve/pii/138572587590013X},
  urldate = {2021-05-23},
  langid = {english}
}

@misc{williamsii2024actionosp12npolytensor,
  title = {{Action of $\mathfrak{osp}(1|2n)$ on polynomials tensor $\mathbb{C}^{0|2n}$}},
  author = {Williams II, Dwight Anderson},
  year = {2024},
  month = nov,
  number = {arXiv:2408.12324},
  eprint = {2408.12324},
  primaryclass = {math.RT},
  publisher = {arXiv},
  doi = {https://arxiv.org/abs/2408.12324},
  url = {https://arxiv.org/abs/2408.12324},
  archiveprefix = {arXiv},
  copyright = {All rights reserved}
}

@phdthesis{williamsiiBasesInfinitedimensionalRepresentations2020,
  title = {Bases of Infinite-Dimensional Representations of Orthosymplectic {{Lie}} Superalgebras},
  author = {Williams II, Dwight Anderson},
  year = {2020},
  address = {Arlington, TX},
  url = {https://mavmatrix.uta.edu/math_dissertations/169/},
  school = {The University of Texas at Arlington}
}

@article{zhelobenkoHypersymmetriesExtremalEquations1997,
  ids = {zhelobenkoHypersymmetriesExtremalEquation,zhelobenkoHypersymmetriesExtremalEquationsa,zhelobenkoHypersymmetriesExtremalEquationsb},
  title = {Hypersymmetries of {{Extremal Equations}}},
  author = {Zhelobenko, D. P.},
  year = {1997},
  journal = {Nova Journal of Theoretical Physics},
  volume = {5},
  number = {4},
  pages = {243--258},
  publisher = {Citeseer}
}

@book{zhelobenkoRepresentationsReductiveLie1994,
  author    = {Zhelobenko, D. P.},
  title     = {Representations of Reductive {L}ie Algebras},
  publisher = {Nauka},
  address   = {Moscow},
  year      = {1994},
  note      = {In Russian}
}

@book{lawsonSpinGeometry1989,
  author    = {Lawson, Jr., H. Blaine and Michelsohn, Marie-Louise},
  title     = {Spin Geometry},
  series    = {Princeton Mathematical Series},
  volume    = {38},
  publisher = {Princeton University Press},
  address   = {Princeton, NJ},
  year      = {1989}
}

\end{document}